\documentclass{amsart}
\usepackage{amscd,amsmath,amssymb,amsfonts}
\usepackage[cmtip, all]{xy}

\newtheorem{thm}{Theorem}[section]
\newtheorem{prop}[thm]{Proposition}
\newtheorem{lem}[thm]{Lemma}
\newtheorem{cor}[thm]{Corollary}
\newtheorem{conj}[thm]{Conjecture}

\renewcommand{\theclaim}{\kern-3pt}

\newtheorem{IntroThm}{Theorem}
\newtheorem{IntroCor}[IntroThm]{Corollary}

\theoremstyle{definition}
\newtheorem{Def}[thm]{Definition}

\theoremstyle{remark}
\newtheorem{rem}[thm]{Remark}

\newtheorem{IntroRems}{Remarks}

\numberwithin{equation}{section}

\newcommand{\sA}{{\mathcal A}}
\newcommand{\sB}{{\mathcal B}}
\newcommand{\sC}{{\mathcal C}}
\newcommand{\sD}{{\mathcal D}}
\newcommand{\sE}{{\mathcal E}}
\newcommand{\sF}{{\mathcal F}}

\newcommand{\sH}{{\mathcal H}}

\newcommand{\sL}{{\mathcal L}}
\newcommand{\sM}{{\mathcal M}}

\newcommand{\sO}{{\mathcal O}}

\newcommand{\sR}{{\mathcal R}}
\newcommand{\sS}{{\mathcal S}}
\newcommand{\sT}{{\mathcal T}}

\newcommand{\sX}{{\mathcal X}}
\newcommand{\sY}{{\mathcal Y}}

\newcommand{\A}{{\mathbb A}}

\newcommand{\C}{{\mathbb C}}

\newcommand{\G}{{\mathbb G}}

\newcommand{\N}{{\mathbb N}}
\renewcommand{\P}{{\mathbb P}}
\newcommand{\Q}{{\mathbb Q}}

\newcommand{\mS}{{\mathbb S}}

\newcommand{\Z}{{\mathbb Z}}

\renewcommand{\phi}{\varphi}

\renewcommand{\1}{{\mathbf{1}}}

\newcommand{\an}{{\rm an}}

\newcommand{\red}{{\rm red}}

\newcommand{\Hom}{{\rm Hom}}

\newcommand{\Ext}{{\rm Ext}}

\newcommand{\Spec}{\operatorname{Spec}}

\newcommand{\0}{\emptyset}
\newcommand{\sHom}{{\mathcal{H}{om}}}

\newcommand{\id}{{\operatorname{id}}}

\newcommand{\Sch}{{\operatorname{\mathbf{Sch}}}}

\newcommand{\Top}{{\mathbf{Top}}}

\newcommand{\op}{{\text{\rm op}}}
\newcommand{\<}{\mathopen<}
\renewcommand{\>}{\mathclose>}

\newcommand{\del}{\partial}

\newcommand{\Spt}{{\mathbf{Spt}}}
\newcommand{\Spc}{{\mathbf{Spc}}}
\newcommand{\Sm}{{\mathbf{Sm}}}

\newcommand{\hocolim}{\mathop{{\rm hocolim}}}
\renewcommand{\lim}{\operatornamewithlimits{\varprojlim}}
\newcommand{\colim}{\operatornamewithlimits{\varinjlim}}

\newcommand{\Ho}{{\mathbf{Ho}}}

\newcommand{\sq}{\square}
 \newcommand{\Ab}{{\mathbf{Ab}}}
\newcommand{\Tot}{{\operatorname{\rm Tot}}}

\newcommand{\Sym}{{\operatorname{\rm Sym}}}
\newcommand{\fin}{{\operatorname{\rm fin}}}
\newcommand{\SH}{{\operatorname{\sS\sH}}}

\newcommand{\eff}{{\mathop{eff}}}

\newcommand{\DM}{{DM}}
\newcommand{\GW}{\operatorname{GW}}

\newcommand{\ds}{{/\kern-3pt/}}

\newcommand{\Cor}{{\mathop{\rm{Cor}}}}

\newcommand{\Deg}{{\mathop{\rm{deg}}}}

\newcommand{\EM}{{{EM}_{\A^1}}}

\newcommand{\gm}{{\mathop{gm}}}

\newcommand{\Mod}{{\operatorname{Mod}}}

\renewcommand{\:}{\kern-1.5pt:\kern-1.5pt}
\newcommand{\tr}{{tr}}

\newcommand{\gr}{{\text{Gr}}}

\renewcommand{\Re}{{\mathop{Re}}}
\newcommand{\ChSS}{{\mathop{\bf{ChSS}}}}

\newcommand{\MGL}{{\mathop{MGL}}}
\renewcommand{\L}{\mathbb{L}}

\newcommand{\ind}{\text{ind}}

\newcommand{\pt}{pt}

\begin{document}

\title{A comparison of motivic and classical stable homotopy theories}
\author{Marc Levine}
\address{Universit\"at Duisburg-Essen\\
Fakult\"at Mathematik, Campus Essen\\
45117 Essen\\
Germany}
\email{marc.levine@uni-due.de}
\thanks{Research supported by the Alexander von Humboldt Foundation}


 
\renewcommand{\abstractname}{Abstract}
\begin{abstract}  
Let $k$ be an algebraically closed field of characteristic zero. Let  $c:\SH\to\SH(k)$ be the functor induced by sending a space to the constant presheaf of spaces on $\Sm/k$. We show that $c$ is fully faithful. In consequence,  $c$ induces an isomorphism
\[
c_*:\pi_n(E)\to \Pi_{n,0}(c(E))(k) 
\]
for all spectra $E$ and all $n\in\Z$.

Fix an embedding $\sigma:k\to \C$ and let $\Re_B:\SH(k)\to \SH$ be the associated Betti realization.  We show that the slice tower for the motivic sphere spectrum over $k$, $\mS_k$ has Betti realization which is strongly convergent. This gives a spectral sequence ``of  motivic origin" converging to the homotopy groups of the sphere spectrum  $\mS\in \SH$; this spectral sequence at $E_2$ agrees with the $E_2$ terms in the Adams-Novikov spectral sequence after a reindexing. Finally, we show that, for $\sE$ a torsion object in $\SH(k)^\eff$, the Betti realization induces an isomorphism $\Pi_{n,0}(\sE)(k)
\to \pi_n(\Re_B\sE)$ for all $n$, generalizing the Suslin-Voevodsky theorem comparing mod $N$ Suslin homology and mod $N$ singular homology.
 \end{abstract}
\date{\today}
\maketitle
\tableofcontents

\section*{Introduction}  
Our main object in this paper is to use Voevodsky's slice tower \cite{VoevSlice}   and its Betti realization to prove two comparison results relating the classical stable homotopy category $\SH$ and the motivic version $\SH(k)$, for $k$ an algebraically closed field of characteristic zero.

For $\sE\in \SH(k)$, we have the bi-graded homotopy sheaf $\Pi_{a,b}\sE$, which is the Nisnevich sheaf on $\Sm/k$ associated to the presheaf
\[
U\mapsto [\Sigma^a_{S^1}\Sigma^b_{\G_m}\Sigma^\infty_TU_+,\sE]_{\SH(k)}
\]
(note the perhaps non-standard indexing).

Our first result concerns the exact symmetric monoidal functor
\[
c:\SH\to \SH(k).
\]
The functor $c$ is derived from the constant presheaf functor from pointed spaces   to presheaves of pointed spaces over $\Sm/k$. It is not hard to show that $c$ is faithful for $k$ an arbitrary characteristic zero field\footnote{If $k$ admits an embedding in $\C$, the corresponding Betti realization gives a left splitting to $c$. In general, one may use a limit argument, relying on  \cite[proposition A.1.2]{AyoubRigid}.}. We will improve this by showing

\begin{IntroThm}\label{IntroThm:Main} Let $k$ be an algebraically closed field of characteristic zero. Then the constant presheaf functor
$c:\SH\to \SH(k)$
is fully faithful.
\end{IntroThm}
As a special case, theorem~\ref{IntroThm:Main} implies

\begin{IntroCor} \label{IntroCor:Main}  Let $k$ be an algebraically closed field of characteristic zero. Let $\mS_k$ be the motivic sphere spectrum in $\SH(k)$ and $\mS$ the classical sphere spectrum in $\SH$. Then the constant presheaf functor induces an isomorphism
\[
c: \pi_n(\mS)\to \Pi_{n,0}\mS_k(k)
\]
for all $n\in\Z$.
\end{IntroCor}
In fact, the corollary implies the theorem, by a density argument (see lemma~\ref{lem:FFReduction}).

\begin{IntroRems} 1. As pointed out by the referee, the functor $c$ is induced by a (left) Quillen functor between  model categories (see the proof of lemma~\ref{lem:FFReduction}), so we do achieve a comparison of ``homotopy theories", as stated in the title, rather than just the underlying homotopy categories.
\\
2. The functor $c$ is {\em not} full in general. In fact, for a perfect field $k$, Morel  \cite[lemma 3.10, corollary 6.43]{MorelA1} has constructed an isomorphism of $\Pi_{0,0}\mS_k(k)$ with the Grothendieck-Witt group $\GW(k)$ of  symmetric bilinear forms over $k$. As long as not every element of $k$ is a square, the augmentation ideal in $\GW(k)$ is non-zero, hence $c:\pi_0(\mS)\to \Pi_{0,0}\mS_k(k)$ is not surjective. Of course, if $k$ is algebraically closed, then $\GW(k)=\Z$ by rank, and thus $c:\pi_0(\mS)\to \Pi_{0,0}\mS_k(k)$ is an isomorphism. This observation can be viewed as the starting point for our main result.
\end{IntroRems}

We have as well a homotopy analog of the theorem of Suslin-Voevodsky comparing Suslin homology and singular homology with mod $N$ coefficients \cite[theorem 8.3]{SuslinVoev}:

\begin{IntroThm}\label{IntroThm:Main2} 
 Let $k$ be an algebraically closed field of characteristic zero with an embedding $\sigma:k\hookrightarrow \C$. Then for all $X\in \Sm/k$, all  $N>1$ and $n\in\Z$,  the Betti realization associated to $\sigma$ induces an isomorphism
\[
\Pi_{n,0}(\Sigma^\infty_TX_+;\Z/N)(k)\cong \pi_n(\Sigma^\infty X^\an_+;\Z/N).
\]
\end{IntroThm}
See corollary~\ref{cor:SV} for a more general statement.

The idea for the proof of theorem~\ref{IntroThm:Main}  is as follows:  As mentioned above, we reduce by a density argument to proving corollary~\ref{IntroCor:Main}; a limit argument reduces us to the case of an algebraically closed field  admitting an embedding into $\C$.  We consider Voevodsky's {\em slice tower} for the sphere spectrum
\[
\ldots\to f_{n+1}\mS_k\to f_n\mS_k\to\ldots\to f_0\mS_k=\mS_k
\]
and its Betti realization. Let $s_n\mS_k$ be the $n$th layer in this tower. This gives us a spectral sequence starting with $\Pi_{*,0}s_n\mS_k(k)$, which should converge to  $\Pi_{*,0}\mS_k(k)$. Similarly, we have a spectral sequence starting with $\pi_*(\Re_B^\sigma(s_n\mS_k))$, which should converge to $\pi_*\mS$ (since $\Re_B^\sigma(\mS_k)=\mS$). By a theorem of Pelaez \cite{Pelaez}, the layers $s_n\mS_k$ are effective motives. Some computations found in our paper \cite{LevineConv} show that $s_n\mS_k$ is in fact a torsion effective motive for $n>0$. On the other hand, Voevodsky \cite{VoevS0} has computed the 0th layer $s_0\mS_k$, and shows that this is the motivic Eilenberg-MacLane spectrum $M\Z$. The theorem of Suslin-Voevodsky {\it loc.\,cit.}\! shows that the Betti realization associated to an embedding $k\hookrightarrow\C$ gives an isomorphism from  the Suslin homology of a  torsion effective motive to the singular homology of its Betti realization; one handles the 0th slice by a direct computation. 

To complete the argument, it suffices to show that the two spectral sequences are strongly convergent. The strong convergence of the motivic version was settled in \cite{LevineConv}, so the main task in this paper is to show that the Betti realization of the slice tower also yields a strongly convergent spectral sequence. 

We accomplish this by introducing a second truncation variable into the story, namely we consider a motivic version of the classical Postnikov tower, filtering by ``topological connectivity". Our results along this line can be viewed as a refinement of Morel's construction of the homotopy $t$-structure on $\SH(k)$ \cite{MorelLec}. In fact, Morel's $\A^1$-connectedness theorem shows  that $\Pi_{a,b}\mS_k=0$ for $a<0$, $b\in \Z$. Our extension of this is our result that this same connectedness in the topological variable $a$ passes to all the terms $f_n\mS_k$ in the slice tower (this is of course a general phenomenon, not restricted to the sphere spectrum, see proposition~\ref{prop:SliceConn}(1)). 

In order to translate this connectedness in the homotopy sheaves into connectedness in the Betti realization, we adapt the method employed by Pelaez in \cite{Pelaez}, using the technique of right Bousfield localization. Using this approach, we are able to show that the $f_n\mS_k$ are built out of objects of the form $\Sigma^a_{S^1}\Sigma^b_{\G_m}\Sigma^\infty_T X_+$ with $b\ge n$ and $a\ge0$ (and $X\in\Sm/k$). As both $\G_m$ and $S^1$ realize to $S^1$, this shows that $f_n\mS_k$ has Betti realization which is $n-1$ connected. 

The proof of theorem~\ref{IntroThm:Main2} runs along the same lines as that of theorem~\ref{IntroThm:Main}, except that we start from the beginning with a torsion object, so we omit the {\it ad hoc} computation of the 0th layer that occurs in the proof of theorem~\ref{IntroThm:Main}.

We conclude  the paper with a closer look at the layers in the slice tower for $\mS_k$.  Voevodsky has given a conjectural formula for these, generalizing his computation of $s_0\mS_k$. The conjecture gives a connection of the layer $s_q\mS_k$ with the complex of homotopy groups (in degree $2q$) arising from the Adams-Novikov spectral sequence. Relying on a result of Hopkins-Morel (see  the preprint of M. Hoyois \cite{Hoyois}), we give a sketch of the proof of Voevodsky's conjecture. 

Via our main result, the Betti  realization of the slice tower for $\mS_k$ gives a tower converging to $\mS$ in $\SH$.  Voevodsky's conjecture shows that the associated spectral sequence converging to the homotopy groups of $\mS$ has $E_2$-term closely related to the $E_2$-terms in the Adams-Novikov spectral sequence.  Our results and Voevodsky's conjecture lead to  the following:

\begin{IntroThm}\label{IntroThm:Main3} Let $k$ be an algebraically closed field of characteristic zero. Let $E_2^{p,2q}(AN)$ be the $E^{p,2q}_2$ term in the Adams-Novikov spectral sequence, i.e., 
\[
E_2^{p,2q}(AN)=\Ext^{p, -2q}_{MU_*(MU)}(MU_*, MU_*),
\]
and let $E_2^{p,q}(AH)$ be the $E_2^{p,q}$ term in the ``Atiyah-Hirzebruch" spectral sequence for $\Pi_{*,0}\mS_k(k)$ associated to the slice tower for $\mS_k$, i.e.,
\[
E_2^{p,q}(AH)=\Pi_{-p-q,0}(s_{-q}\mS_k)(k)\Longrightarrow\Pi_{-p-q,0}\mS_k(k)=\pi_{-p-q}(\mS),
\]
Then
\[
E_2^{p,q}(AH)=E_2^{p-q,2q}(AN)\otimes\hat{\Z}(q),
\]
where $\hat{\Z}(q)=\lim_N\mu_N^{\otimes q}$.
\end{IntroThm}
\noindent 
See theorem~\ref{thm:SliceAN} for the details and proof of this result. 

It would be interesting to see if there were a deeper connection relating the Atiyah-Hirzebruch spectral sequence (for $k=\bar k$ of characteristic zero) and the Adams-Novikov spectral sequence via our theorem~\ref{IntroThm:Main} identifying $\Pi_{-p-q,0}(\mS_k)(k)$ with $\pi_{-p-q}(\mS)$. Although the Betti realization of the slice tower for $\mS_k$ gives a tower converging to $\mS$ in $\SH$ and the associated spectral sequence converging to the homotopy groups of $\mS$ has $E_2$ term the same (up to reindexing) as the $E_2$-terms in the Adams-Novikov spectral sequence,  we do not know if the two spectral sequences continue to be the same. Taking into account the reindexing in comparing the $E_2$-terms,   we raise the question: is $E_r^{p,q}(AH)=E_{2r-1}^{p-q,-2q}(AN)\otimes\hat\Z(q)$ and $d_r^{p,q}(AH)=d_{2r-1}^{p-q,-2q}(AN)\otimes\id$ for all $r\ge2$?

Dugger and Isaksen \cite{DI}  and independently Hu, Kriz and Ormsby \cite{HuKrizOrmsby} have constructed motivic versions of the Adams and  Adams-Novikov spectral sequences, and have made explicit computations. For $k$ algebraically closed,  the work of  \cite{DI}   and  \cite{HuKrizOrmsby} shows that  the Betti realization gives an isomorphism of the 2-completed weight 0 parts of the motivic  Adams, resp. motivic Adams-Novikov, spectral sequence with their topological counterpart.  It would be interesting to see what deeper connections the slice tower for $\mS_k$ has with the motivic Adams or motivic Adams-Novikov spectral sequences, not just for the case of algebraically closed fields. 

As the slice tower has a model based on the filtration by codimension of support on the cosimplicial algebraic simplex $\Delta^*$, such a connection could introduce a new point of view for studying the both the motivic as well as the classical Adams-Novikov spectral sequences. In particular, we find it intriguing  that the Adams-Novikov level of an element in the stable homotopy group of spheres could have a corresponding codimension of support coming from the slice spectral sequence, even though the sphere spectrum itself has no evident algebro-geometric structure. Conversely, the interesting algebraic structure enjoyed by the $E_2$-term of the Adams-Novikov sequence as an Ext group over the co-algebra of co-operations on   $MU$ is not immediately apparent in the layers of the slice tower. 

The paper is organized as follows. The first three sections deal with the construction of the two-variable Postnikov tower and a discussion of its properties. In \S \ref{sec:Cellular} we recall some of the background on cofibrantly generated and cellular model categories. In \S \ref{sec:RBLoc} we discuss some facts about right Bousfield localization and we apply this machinery to give the construction of the two-variable tower in \S \ref{sec:Postnikov2}.  We prove our main connectedness results in \S \ref{sec:Connected}.   We recall some facts about the  Betti realization in \S \ref{sec:Betti}, prove our main theorem on the connectedness of the Betti realization (theorem~\ref{thm:Connected}) and make a few simple computations. We also describe the consequences of the Suslin-Voevodsky theorem for torsion effective motives and their Betti realizations (corollary~\ref{cor:SV}).  

The next  two sections,  \S \ref{sec:Proof}  and \S \ref{sec:SVHpty}, assemble all the pieces to prove theorems~\ref{IntroThm:Main}  and \ref{IntroThm:Main2}. We conclude the body of the paper with a discussion of Voevodsky's conjecture on the slices of the sphere spectrum in \S \ref{sec:VoevConj}. In an appendix, we collect some results on symmetric products that are needed for our study of the Betti realization; although these results closely parallel discussions of symmetric products already in the literature (for example \cite{VoevMotivic}), we found it difficult to derive exactly what we need from these existing treatments. 

I would like to thank Ivan Panin for discussions that encouraged me to look at the possibility of extending the Suslin-Voevodsky theorem to  the Betti realization for $\SH(k)$. I would also like to thank Pablo Pelaez for discussing aspects of Bousfield localization with me and  pointing out that this is an effective way of defining Postnikov towers. Thanks are  also due to Daniel Dugger, Javier Guti\'errez, Shane Kelly,  Oliver R\"ondigs and Markus Spitzweck, as well as to  the referee,  for a number of very helpful comments and suggestions.

\section{Cellular model structures}\label{sec:Cellular} In  section \ref{sec:Postnikov2}, we apply the method used by Pelaez \cite{Pelaez}, in his study of the slice filtration in $\SH(k)$, to define a two-variable Postnikov tower in $\SH(k)$. The method relies on the fact that  motivic model structure on  $\Spt_T(k)$ is cellular, which allows one to take a right Bousfield localization. In this section, we recall the basic facts concerning the cellularity of  $\Spt_T(k)$ and some other auxiliary model categories.

See \cite[definition 11.1.2]{Hirschhorn}  for the definition  of a cofibrantly generated model category and \cite[definition 12.1.1]{Hirschhorn} for that of a cellular model category.  For the complete story, we refer the reader to  \cite{Hirschhorn};  
an earlier version of this paper \cite{LevineArXiv} also contains some additional details omitted here.

The category  of simplicial presheaves on $\Sm/k$, $\Spc(k)$, and the category of pointed simplicial presheaves,  $\Spc_\bullet(k)$, have {\em motivic model structures}; we denote these model categories by $\sM(k)$, $\sM_\bullet(k)$. $\sM(k)$, $\sM_\bullet(k)$ are proper simplicial symmetric monoidal cellular model category with respective homotopy categories  the Morel-Voevodsky   unstable motivic homotopy categories, $\sH(k)$,  $\sH_\bullet(k)$ \cite{MorelVoev} (see e.g., \cite[corollary 1.6]{Hornbostel}, \cite[\S1, theorem 1.1]{Jardine}, \cite[Appendix A]{Jardine2} and \cite[theorem 2.3.2]{Pelaez} for details, including the definition of the generating cofibrations $I_{\sM}$, and  generating trivial cofibrations $J_\sM$).

We pass to the stable setting. Let $T=S^1\wedge\G_m$ and let  $\Spt_T(k)$ be the category of $T$-spectra in 
$\Spc_\bullet(k)$, i.e., objects are sequences $\sE:=(E_0, E_1,\ldots, E_n,\ldots)$, $E_n\in  \Spc_\bullet(k)$,  together with bonding maps $\epsilon_n:E_n\wedge T\to E_{n+1}$. Morphisms are sequences of maps compatible with the bonding. One defines the notion of a {\em stable $\A^1$ weak equivalence}  $f:\sE\to \sF$; see for example \cite[pg. 470]{Jardine2}.

\begin{thm}[\hbox{\cite[theorem 2.5.4]{Pelaez}}]\label{thm:Cellular} There is a cellular model structure, $\sM\sS_T(k)$ on $\Spt_T(k)$ such that the weak equivalences are the stable $\A^1$ weak equivalences. With this model structure, $\sM\sS_T(k)$  is a proper simplicial $\sM_\bullet(k)$ model category.
\end{thm}
An explicit description of the generating cofibrations and generating trivial cofibrations is given in the statement of \cite[theorem 2.5.4]{Pelaez}.

\begin{rem} \label{rem:ModelStructure} The model structure $\sM\sS_T(k)$ is the one defined by Jardine in \cite[theorem 2.9]{Jardine2}; replacing $T$ with $\P^1$ or $S^1$ gives the {\em motivic model structure} for $\P^1$-spectra $\Spt_{\P^1}(k)$ or $S^1$-spectra, 
$\Spt_{S^1}(k)$.

One may use the  argument for  \cite[theorem 2.5.4]{Pelaez}, replacing $\G_m$ with $S^0$ throughout, to show that motivic model structure on  
$\Spt_{S^1}(k)$ defines a cellular proper simplicial $\sM_\bullet(k)$ model category.
\end{rem}

We recall that a model category $\sM$ is {\em combinatorial} if $\sM$ is cofibrantly generated and locally presentable \cite[definition 2.1]{Dugger}.
 We will have occasion to use functor categories $\sM^\sC$ for $\sM$ a model category and $\sC$ a small category. For this, we recall the following:
\begin{prop} \label{prop:FunctorModelCat} For $f:F\to G$ a morphism in $\sM^\sC$, define $f$ to be a fibration (resp. a weak equivalence) if $F(c)\to G(c)$ is a fibration (resp. a weak equivalence) for all $c\in\sC$; $f$ is a cofibration if it has the LLP with respect to trivial fibrations.  Suppose that  $\sM$ is a cofibrantly generated, resp.  cellular, resp. combinatorial model category.
Then with these cofibrations, fibrations and weak equivalences, $\sM^\sC$ is  model category;  $\sM^\sC$ is cofibrantly generated, resp.,  cellular, resp. combinatorial. If $\sM$ is left, resp. right, proper, the same holds for $\sM^\sC$. 
\end{prop}

\begin{proof} See    \cite[theorem 11.6.1, proposition 12.1.5]{Hirschhorn} and \cite[theorem 2.14]{Barwick}. 
\end{proof}
This model structure is called the {\em projective model structure} on $\sM^\sC$.

\section{Right Bousfield localization}\label{sec:RBLoc} 
We recall the notions of the left and right  Bousfield localization  of a model category from  \cite[\S 3]{Hirschhorn}. The  machinery of cellular model categories is useful for Bousfield localization due to the following theorem of Hirschhorn:

\begin{thm}[\hbox{\cite[theorems 4.1.1 and 5.1.1]{Hirschhorn}}] Let $\sM$ be a cellular model category. \\
1. Suppose $\sM$ is  left proper. Let $\sS$ be a set of maps in $\sM$. Then the left Bousfield localization of $\sM$ with respect to $\sS$,  $L_\sS\sM$, exists. \\
2. Suppose $\sM$ is  right proper. Let $K$ be a set of objects in $\sM$. Then the right Bousfield localization of $\sM$ with respect to the set $\sC(K)$ of  $K$-local maps in $\sM$, $R_{\sC(K)}\sM$, exists.
\end{thm}
We sometimes abuse notation and write $R_K\sM$ for $R_{\sC(K)}\sM$; we will also call $R_K\sM$ the right Bousfield localization of $\sM$ with respect to $K$.

Let $K$ be a class of objects in a pointed  model category $\sM$. The definition of a {\em $K$-colocal weak equivalence} and a {\em $K$-colocal} object in  $\sM$ is given in \cite[definitions 3.1.4, 3.1.8]{Hirschhorn}.  

\begin{Def}[\hbox{\cite[definition 5.1.4]{Hirschhorn}}] Let $\sM$ be a model category, $K$ a set of cofibrant objects of $\sM$. The class of {\em $K$-cellular objects} is the smallest class of cofibrant objects of $\sM$ containing $K$ and closed under homotopy colimits and weak equivalences.
\end{Def}

\begin{rem}\label{rem:Hocolim} Suppose that $\sM$ is a stable model category, that is,   $\sM$ is a pointed model category such that the suspension functor on $\Ho\sM$ is an auotequivalence \cite[definition 7.1.1]{Hovey}. By \cite[proposition 7.1.6]{Hovey},  $\Ho\sM$ becomes a triangulated category with translation equal to suspension and the distinguished triangles the mapping cone sequences. Let $K$ be a set of cofibrant objects of $\sM$ containing the base-point. Then the image of the class of $K$-cellular objects in $\Ho\sM$ is the class of objects in the smallest full subcategory $\sC$ of $\Ho\sM$ containing $K$,  closed under arbitrary small coproducts and with the property that, if $A\xrightarrow{f} B\to C\to A[1]$ is a distinguished triangle with $A$ and $B$ in $\sC$, then $C$ is in $\sC$ (we call  such a $C$ a {\em cone} of the morphism $f$).

Indeed, each such distinguished triangle exhibits $C$ as the homotopy colimit of $\pt\leftarrow A\to B$. Conversely, if
$F:I\to \sM$
is a functor from a small category $I$ with $F(\alpha)
\in K$ for all $\alpha\in I$, then $\hocolim_IF$ can be expressed as a colimit of a sequence of cofibrations $C_0\to C_1\to\ldots\to C_n\to\ldots$,
with each map $C_n\to C_{n+1}$ given by a pushout diagram
\[
\xymatrix{
\coprod F(\alpha) \otimes S^n\ar[r] \ar[d]&C_n\ar[d]\\
\coprod F(\alpha) \otimes D^{n+1}\ar[r] &C_{n+1}}
\]
with the coproduct over a suitable index set. Thus in  $\Ho\sM$, we have the distinguished triangle
\[
\oplus  F(\alpha)[n]\to C_n\to C_{n+1}\to \oplus_\alpha F(\alpha)[n+1],
\]
hence $C_n$ is in $K$ for each $n$. This gives the distinguished triangle
\[
\oplus_nC_n\to \oplus_nC_n\to \hocolim_IF\to \oplus_nC_n[1].
\]
\end{rem}

\begin{thm}[\hbox{\cite[theorem 5.1.1, theorem 5.1.5]{Hirschhorn}}]\label{thm:RBousLoc} Let $K$ be a set of objects in a right proper cellular   model category $\sM$, $\sR_K\sM$ the right Bousfield localization. \\
1.  $\sR_K\sM$ is a right proper model category; if $\sM$ is a simplicial model category, then  $\sR_K\sM$ inherits the structure of a simplicial model category from $\sM$.\\
2. The cofibrant objects in $\sR_K\sM$ are the $K$-colocal objects of $\sM$.\\
3. If the objects in $K$ are all cofibrant, then the class of $K$-colocal objects is the same as the class of $K$-cellular objects.
\end{thm}

We will be using the properties of right Bousfield localization as expressed in the following result, essentially a direct consequence of theorem~\ref{thm:RBousLoc}.

\begin{thm} \label{thm:truncation} Let $K$ be a set of cofibrant objects in a right proper cellular model category $\sM$ and let $\Ho\sM(K)$ be the full subcategory of $\Ho\sM$ with objects the  $K$-cellular objects of $\sM$. Then
\\
1. the inclusion $i:\Ho\sM(K)\to \Ho\sM$ admits a right adjoint $r:\Ho\sM\to \Ho\sM(K)$.\\
2. For an object $X\in\sM$, $i\circ r(X)$ is the image in $\Ho\sM$ of a cofibrant replacement $A\to X$ with respect to the model structure $\sR_K\sM$.\\
3. $r:\Ho\sM\to \Ho\sM(K)$ identifies $\Ho\sM(K)$ with the localization of $\Ho\sM$ with respect to  the $K$-colocal weak equivalences. This localization is canonically equivalent to the functor $q:\Ho\sM\to \Ho\sR_K\sM$ induced by the identity functor $\sM\to \sR_\sM$ on the underlying category of $\sM$ and $\sR_K\sM$.
\end{thm}

\begin{proof} By theorem~\ref{thm:RBousLoc}, the right Bousfield localization $\sR_K\sM$ exists. The identity functor $\sM\to \sR_K\sM$ is a right Quillen functor with right derived functor the localization $q:\Ho\sM\to \Ho\sR_K\sM$. The left adjoint to $\id:\sM\to \sR_K\sM$ is of course the identity functor $\sR_K\sM\to \sM$; letting $L\id:\Ho\sR_K\sM\to \Ho\sM$ be its left derived functor, it follows directly from the definition of $L\id$ and  theorem~\ref{thm:RBousLoc} that the  image of $L\id$ is  $\Ho\sM(K)$. It is easy to see that the induced functor $\Ho\sR_K\sM\to \Ho\sM(K)$ is an equivalence, proving (1) and (3); (2) follows from the definition of $L\id$.
\end{proof}

\section{A two-variable Postnikov tower}\label{sec:Postnikov2}  Following a suggestion of P. Pelaez, we refine the construction of Voevodsky's slice filtration to a two-variable version which measures both $S^1$-connectedness and $\G_m$-connectedness.

We consider $\Spt_T(k)$ with its motivic model structure $\sM\sS_T(k)$.  For $n\ge0$, let $F_n:\Spc_\bullet(k)\to \Spt_T(k)$ be the functor with 
\[
F_n(\sX):=(F_n(\sX)_0, F_n(\sX)_1,\ldots, F_n(\sX)_m,\ldots), 
\]
where
\[
F_n(\sX)_m:=\begin{cases} \pt&\text{ if } m<n\\ \Sigma_T^{m-n}\sX&\text{ if }m\ge n\end{cases}
\]
The bonding maps $\epsilon_m$ are the identity if $m\ge n$, the basepoint map if $m<n$. 

For integers $a,b$, let
\[
K_{a,b}:=\{F_n(\Sigma^p_{S^1}\Sigma^q_{\G_m}X_+)\ |\ X\in\Sm/k, p-n\ge a, q-n\ge b\}.
\]
We also allow $a=-\infty$ or $b=-\infty$. This gives us the full subcategories of $\SH(k)=\Ho\Spt_T(k)$
\[
\tau^{a,b}\SH(k):=\Ho\Spt_T(k)(K_{a,b}).
\]

Each object in $K_{a,b}$ is cofibrant. As $F_n(\Sigma^p_{S^1}\Sigma^q_{\G_m}X_+)\cong\Sigma^{p-n}_{S^1}\Sigma^{q-n}_{\G_m}\Sigma_T^\infty X_+$ in $\SH(k)$,  remark~\ref{rem:Hocolim} tells us that $\tau^{a,b}\SH(k)$ is the smallest full subcategory of $\SH(k)$ containing the set of objects $\{\Sigma^p_{S^1}\Sigma^q_{\G_m}\Sigma_T^\infty X_+\ |\ X\in\Sm/k, p\ge a, q\ge b\}$, closed under small coproducts and taking cones of morphisms.  

In addition, $\tau^{-\infty,b}\SH(k)$ is closed under  $\Sigma_{S^1}^n$ for $n\in\Z$, hence  $\tau^{-\infty,b}\SH(k)$ is a localizing subcategory of $\SH(k)$; indeed, $\tau^{-\infty,b}\SH(k)$ is the localizing category generated by the objects $\{\Sigma^q_T\Sigma_T^\infty X_+\ |\ X\in\Sm/k,  q\ge b\}$, which category is used to define Voevodsky's slice tower in $\SH(k)$.

Write $\bar\Z$ for $\Z\cup\{-\infty\}$. Giving $\bar\Z^2$ the partial order $(a,b)\le (a',b')$ iff $a\le a'$ and $b\le b'$, we have
\[
\tau^{a',b'}\SH(k)\subset \tau^{a,b}\SH(k) \text{ if } (a,b)\le (a',b').
\]
Let $i_{a,b}:\tau^{a,b}\SH(k)\to \SH(k)$
be the inclusion.

\begin{thm}\label{thm:PostnikovTower}  For each $(a,b)\in \bar\Z^2$, the inclusion functor $i_{a,b}$ admits a right adjoint $r_{a,b}:\SH(k)\to \tau^{a,b}\SH(k)$. In addition
\\
1. $r_{a,b}$ identifies $\tau^{a,b}\SH(k)$ with the localization of $\SH(k)$ with respect to the $K_{a,b}$-colocal weak equivalences.
\\
2. For $(a,b)\le (a',b')$, the inclusion $i_{a',b'}^{a,b}:\tau^{a',b'}\SH(k)\to \tau^{a,b}\SH(k)$ admits a right adjoint
 $r_{a',b'}^{a,b}:\tau^{a,b}\SH(k)\to \tau^{a',b'}\SH(k)$. There is a canonical isomorphism $r_{a',b'}\cong 
r_{a',b'}^{a,b}\circ r_{a,b}$ and $r_{a',b'}^{a,b}$ identifies $\tau^{a',b'}\SH(k)$ with the localization of $\tau^{a,b}\SH(k)$ with respect to the $K_{a',b'}$-colocal weak equivalences.
\\
3. A morphism $f:X\to Y$ in $\tau^{a,b}\SH(k)$ is an isomorphism if and only if for each $A\in K_{a,b}$ the map $f_*:[A, X]_{\SH(k)}\to [A, Y]_{\SH(k)}$ is an isomorphism.
\end{thm}

\begin{proof} The fact that $i_{a,b}$ admits a right adjoint $r_{a,b}$ and that $r_{a,b}$ identifies $\tau^{a,b}\SH(k)$ with the localization of $\SH(k)$ with respect to the $K_{a,b}$-colocal weak equivalences follows directly from  theorem~\ref{thm:Cellular}, theorem~\ref{thm:RBousLoc} and theorem~\ref{thm:truncation}.   For (2),  as $i_{a,b}$ is fully faithful, it follows that $r_{a',b'}\circ i_{a,b}$ is right adjoint to $i_{a',b'}^{a,b}$. The existence of a canonical isomorphism $r_{a',b'}\cong r_{a',b'}^{a,b}\circ r_{a,b}$ follows directly from the universal property of adjoints. Since $K_{a',b'}\subset K_{a,b}$,  every $K_{a,b}$-colocal weak equivalence is a $K_{a',b'}$-colocal weak equivalence. This together with (1) yields  the last statement in (2).

For (3), the isomorphisms in $\tau^{a,b}\SH(k)$ are given by the $K_{a,b}$-colocal weak equivalences in $\sR_{K_{a,b}}\Spt_T(k)$. Choosing fibrant-cofibrant replacements $\tilde{X}$, $\tilde{Y}$ for $X,Y$, and lifting $f$ to a map $\tilde{f}:\tilde{X}\to \tilde{Y}$, it suffices to show that if 
$f_*:[A, X]_{\SH(k)}\to [A, Y]_{\SH(k)}$
is an isomorphism for all $A\in K_{a,b}$, then $\tilde{f}$ is a $K_{a,b}$-colocal weak equivalence. But if $A$ is in $K_{a,b}$, so is $\Sigma^n_{S^1}A$ for all $n\ge0$. We have
\begin{align*}
\pi_n(\sHom(A,\tilde{X}))&=[A[n], X]_{\SH(k)}\\
&\cong [A[n], Y]_{\SH(k)}\\
&=\pi_n(\sHom(A,\tilde{Y})),
\end{align*}
that is, $f_*:\pi_n(\sHom(A,\tilde{X}))\to \pi_n(\sHom(A,\tilde{Y}))$
is an isomorphism for all $n\ge0$. Thus $f_*:\sHom(A,\tilde{X}))\to \sHom(A,\tilde{Y}))$ is a simplicial weak equivalence for all $A\in K_{a,b}$ and hence $f$ is a  $K_{a,b}$-colocal weak equivalence.
\end{proof}

For $(a,b)\in\bar{Z}^2$, define the endofunctor
\[
f_{a,b}:\SH(k)\to \SH(k)
\]
as the composition $i_{a,b}\circ r_{a,b}$. We write $f^t_n$ for $f_{-\infty, n}$ and $f^s_n$ for $f_{n,-\infty}$. By theorem~\ref{thm:PostnikovTower}(2), we have the lattice of natural transformations
\[
\xymatrixrowsep{2pt}
\xymatrix{
&\vdots\ar[d]&\vdots\ar[d]&&\vdots\ar[d]\\
\ldots\ar[r]&f_{a+1,b+1}\ar[r]\ar[dd]&f_{a,b+1}\ar[dd]\ar[r]&\ldots\ar[r]&f^t_{b+1}\ar[dd]\\\\
\ldots\ar[r]&f_{a+1,b}\ar[r]\ar[dd]&f_{a,b}\ar[dd]\ar[r]&\ldots\ar[r]&f^t_{b}\ar[dd]\\\\
&\vdots\ar[d]&\vdots\ar[d]&&\vdots\ar[d]\\
\ldots\ar[r]&f^s_{a+1}\ar[r]&f^s_a\ar[r]&\ldots\ar[r]&\id}
\]

\begin{rem} In number of papers on Voevodsky's slice tower, one considers the sequence of localizing subcategories
\[
\ldots\subset \tau^{-\infty,b}\SH(k)\subset\tau^{-\infty,b}\SH(k)\subset\ldots\subset\SH(k),
\]
($\tau^{-\infty,b}\SH(k)$ was usually denoted $\Sigma^b_T\SH^\eff(k)$, with $\SH^\eff(k):=\tau^{-\infty,0}\SH(k)$).  The existence of the right adjoint to the inclusion $\Sigma^b_T\SH^\eff(k)\to \SH(k)$ follows from Neeman's Brown representability theorem \cite[theorem 8.3.3]{Neeman}. What we are now writing as $f^t_b$ was usually denoted $f_b$. 

Pelaez \cite{Pelaez} introduced the approach via right Bousfield localization to better understand multiplicative properties of Voevodsky's slice tower.  
\end{rem}

\begin{rem}[The $S^1$ and unstable theory] The above construction goes through with minor changes if we replace $\SH(k)$ with $\SH_{S^1}(k)$. For $n\ge0$, let $F_n^s:\Spc_\bullet(k)\to \Spt_{S^1}(k)$ be the functor
\[
F_n^s(\sX)=(F_n^s(\sX)_0, \ldots, F_n^s(\sX)_m,\ldots)
\]
with $F^s_n(\sX)_m=\pt$ for $m<n$, $ F_n^s(\sX)_m=\Sigma^{m-n}_{S^1}\sX$ for $m\ge n$ and with identity bonding maps.  For $a\in\Z$, $b\ge0$, let $K_{a,b}^s$ be the set of objects $F_n^s(\Sigma^p_{S^1}\Sigma^q_{\G_m}X_+)$, with $p\ge a+n$, $q\ge b$. With its motivic model structure,  $\Spt_{S^1}(k)$  is a cellular proper simplicial model category (see remark~\ref{rem:ModelStructure}) and we can form the right Bousfield localizations $\sR_{K^s_{a,b}}\Spt_{S^1}$. This gives us the subcategories $\tau_s^{a,b}\SH_{S^1}(k)$ of $K^s_{a,b}$-colocal objects, adjoint functors $i^s_{a,b}$, $r^s_{a,b}$  $f^s_{a,b}$, $a\in\Z$, $b\ge0$, with properties exactly analogous to those listed in theorem~\ref{thm:PostnikovTower}. Defining the truncation functors $f^s_{a,b}:=i^s_{a,b}\circ r^s_{a,b}$ gives us the two-variable Postnikov tower in $\SH_{S^1}(k)$.

As  $\Spc_\bullet(k)$ with its motivic model structure is also a cellular proper simplicial model category, the same approach, with 
\[
K_{a,b}^{un}:=\{\Sigma^{a'}_{S^1}\Sigma^{b'}_{\G_m}X_+\ |\ a'\ge a, b'\ge b, X\in \Sm/k\}, a,b\ge0,
\]
defines a two variable Postnikov (really Whitehead) tower in $\sH_\bullet(k)$, again with  properties   analogous to those listed in theorem~\ref{thm:PostnikovTower}.
\end{rem}

\section{Connectedness} \label{sec:Connected}

\begin{Def} Let $\sE\in\SH(k)$. We say that $\sE$ is {\em topologically $N$-connected} if $\Pi_{a,b}(\sE)=0$ for $a\le N$, $b\in\Z$. For $E\in\SH_{S^1}(k)$, we call $E$  topologically $N$-connected if $\Pi_{a,b}(E)=0$ for $a\le N$, $b\ge0$. 
\end{Def}

\begin{rem} An $S^1$-spectrum $E$ is said to be {\em $N$-connected} if the homotopy sheaf $\pi_m E$ is zero for all $m\le N$.  Take $\sE\in \SH(k)$ and let $E_m=\Omega^\infty_{\G_m}(\Sigma^{m}_{\G_m}\sE)\in\SH_{S^1}(k)$. Then $\sE$ is  topologically $N$-connected if and only if $E_m$ is   $N$-connected for all $m\in \Z$. Indeed,  
\[
\pi_a\Omega^\infty_{\G_m}(\Sigma^m_{\G_m}\sE)=\Pi_{a,0}(\Sigma^m_{\G_m}\sE)=\Pi_{a,-m}\sE.
\]
\end{rem}

\begin{lem}\label{lem:SliceIsoPi} Take $\sE\in \SH(k)$ and  $a\ge-\infty$, $b\in\Z$. Then for all $p\ge a$, $q\ge b$, the canonical map $f_{a,b}\sE\to \sE$ induces an isomorphism of sheaves $\Pi_{p,q}f_{a,b}\sE\to \Pi_{p,q}\sE$.
In particular, if $\sE$ is topologically $N$-connected, then for $b\in \Z$,  $\Pi_{p,q}f^t_b\sE=0$ for all $p\le N$, $q\ge b$.
\end{lem}

\begin{proof}  For $U\in \Sm/k$, $\Sigma^p_{S^1}\Sigma^q_{\G_m}\Sigma_T^\infty U_+$ is in $\tau^{a, b}\SH(k)$ for all $p\ge a$, $q\ge b$. Thus
\[
[\Sigma^p_{S^1}\Sigma^q_{\G_m}\Sigma_T^\infty U_+, f_{a,b}\sE]_{\SH(k)}\to 
[\Sigma^p_{S^1}\Sigma^q_{\G_m}\Sigma_T^\infty U_+,\sE]_{\SH(k)}
\]
is an isomorphism for all $p\ge a$, $q\ge b$, by the universal property of $f_{a,b}\sE\to \sE$. Taking the Nisnevich sheaves associated to the presheaves
\[
U\mapsto [\Sigma^p_{S^1}\Sigma^q_{\G_m}\Sigma_T^\infty U_+, f_{a,b}\sE]_{\SH(k)};\
U\mapsto [\Sigma^p_{S^1}\Sigma^q_{\G_m}\Sigma_T^\infty U_+,\sE]_{\SH(k)}
\]
shows that $\Pi_{p,q}f_{a,b}\sE\to \Pi_{p,q}\sE$ is an isomorphism for all $p\ge a$, $q\ge b$, as desired.
\end{proof}
 
\begin{lem}\label{lem:Conn0} Let $X$ be in $\Sm/k$. Then for $p\ge a$, and all $q\in\Z$,    $\Sigma^p_{S^1}\Sigma^q_{\G_m}\Sigma_T^\infty X_+$ is topologically $a-1$ connected. 
 \end{lem} 
 
 \begin{proof} As $\Pi_{a,b}\Sigma^p_{S^1}\Sigma^q_{\G_m}\Sigma_T^\infty X_+=
 \Pi_{a-p, b-q}\Sigma_T^\infty X_+$, we need only show that  $\Sigma_T^\infty X_+$ is topologically -1 connected; this is \cite[proposition 6.9(1)]{LevineConv}.
  \end{proof}
  
 \begin{lem} \label{lem:Conn1} Take $\sE\in\SH(k)$. Then $f_{a,b}\sE$ is topologically $a-1$ connected. 
 \end{lem}
 
 \begin{proof} We need to show that each $\sF$ in $\tau^{a,b}\SH(k)$  is is topologically $a-1$ connected.  By theorem~\ref{thm:truncation}, $\tau^{a,b}\SH(k)$ is the full subcategory of  $K_{a,b}$-colocal objects of $\SH(k)$.  Each element of $K_{a,b}$ is isomorphic in $\SH(k)$ to  $\Sigma^p_{S^1}\Sigma^q_{\G_m}\Sigma^\infty_TX_+$ for some $X\in \Sm/k$, $p\ge a$, $q\ge b$.   By lemma~\ref{lem:Conn0},  $\Sigma^p_{S^1}\Sigma^q_{\G_m}\Sigma_T^\infty X_+$ is topologically $p-1$ connected.

As $\Pi_{a,b}(\oplus_\alpha\sE_\alpha)\cong \oplus_\alpha\Pi_{a,b}\sE_\alpha$ the property  of being  topologically $a-1$ connected is closed under arbitrary coproducts. Similarly, if $\sA$ and $\sB$ are topologically $a-1$ connected and $\sC$ is a cone of a morphism $f:\sA\to \sB$, then
$\sC$ is also topologically $a-1$ connected. But by remark~\ref{rem:Hocolim}, $\tau^{a,b}\SH(k)$ is the smallest full subcategory of 
$\SH(k)$ containing $K_{a,b}$ and closed under taking coproducts and cones of morphisms, and thus each $\sF$ in $\tau^{a,b}\SH(k)$ is 
topologically $a-1$ connected, as desired.
\end{proof}
 
\begin{lem}\label{lem:HomotopySheaf} Let $f:\sF_1\to \sF_2$ be a morphism in $\Sigma^n_T\SH^\eff(k)$. Then $f$ induces an isomorphism
$f_*:\Pi_{a,b}\sF_1\to \Pi_{a,b}\sF_2$
for all $a\in\Z$ and all $b\ge n$ if and only if $f$ is an isomorphism.
\end{lem}

\begin{proof} One  implication is evident. For the other direction,   by theorem~\ref{thm:PostnikovTower}(3), it suffices to show that $f_*: [\Sigma^a_{S^1}\Sigma^m_T\Sigma^\infty_TX_+,\sF_1]_{\SH(k)}\to [\Sigma^a_{S^1}\Sigma^m_{\G_m}\Sigma^\infty_TX_+,\sF_2]_{\SH(k)}$
is an isomorphism for all $X\in\Sm/k$, $a\in \Z$, $m\ge n$. Filtering $X$ by closed subsets of codimension $i$ for $i=0,\ldots, \dim X$ gives the strongly convergent Gersten spectral sequence, with $E_1^{i,j}\neq0$ only for $0\le i\le\dim X$, 
\[
E_1^{i,j}:= \oplus_{x\in X^{(i)}}\Pi_{a-j,m+i}(\sF)(k(x))\Longrightarrow [\Sigma^{a-i-j}_{S^1}\Sigma^m_{\G_m}\Sigma^\infty_TX_+,\sF]_{\SH(k)}.
\]
By assumption, the map $f$ induces an isomorphism on the $E_1$-terms and hence is an isomorphism on the abutment.
\end{proof}
 
\begin{prop} \label{prop:SliceConn} Take $\sE\in\SH(k)$. Suppose that $\sE$ is topologically $N$-connected. Then 
\\
1. For each  $b\in \Z$,  $a\in\Z\cup\{-\infty\}$, $f_{a,b}\sE$ is topologically $N$-connected.
\\
2. For each $b\in \Z$, $a\le N+1$, the canonical map  $f_{a,b}\sE\to f^t_b\sE$
is an isomorphism in $\SH(k)$.
\end{prop}

\begin{proof}  We first prove (2). Since both $f_{a,b}\sE$ and $f^t_b\sE$ are in $\Sigma^b_T\SH^\eff(k)$, it suffices to show that $f_{a,b}\sE\to f^t_b\sE$ is an isomorphism in $\Sigma^b_T\SH^\eff(k)$. Thus (lemma~\ref{lem:HomotopySheaf}) we need only see that
$\Pi_{p,q}f_{a,b}\sE\to \Pi_{p,q}f^t_b\sE$ is an isomorphism for all $p\in\Z$, $q\ge b$. 

By lemma~\ref{lem:SliceIsoPi}, the map $\Pi_{p,q}f_{a,b}\sE\to \Pi_{p,q}f^t_b\sE$ is an isomorphism for all $p\ge a$, $q\ge b$. The fact that $\sE$ is $a-1$-connected together with lemma~\ref{lem:SliceIsoPi} tells us that $ \Pi_{p,q}f^t_b\sE=0$ for $p<a$. By lemma~\ref{lem:Conn1} $f_{a,b}\sE$ is  topologically $a-1$-connected, so $\Pi_{p,q}f_{a,b}\sE=0$ for $p<a$, $q\ge b$ as well, and (2) is proved.

For (1), $f_{a,b}\sE$ is  topologically $a-1$-connected for all $b\in \Z$ (lemma~\ref{lem:Conn1}), so it remains to prove (1) for $a\le N$. By (2) (with $a=N+1$) we see that  $f^t_b\sE$ is topologically $N$-connected for all $b\in \Z$. Applying (2) again with $a\le N$ completes the proof of (1).
\end{proof}

\section{The Betti realization} \label{sec:Betti}  There have been a number of constructions of the Betti realization in varying levels of generality, most recently by  Ayoub \cite[definition 2.1]{Ayoub}, but see also Riou \cite{Riou} and Voevodsky \cite[\S 4]{VoevMotivic}. As we will not need the level of generality provided by Ayoub's construction, we use instead an earlier construction due to Panin-Pimenov-R\"ondigs \cite[\S A4]{PPR}.

Let $\Top$ (resp. $\Top_\bullet$) denote the category of (pointed) compactly generated topological spaces. For $X$ a finite type $k$-scheme, we write $X^\an$ for $X(\C)$ with the classical topology. The realization functor is induced by the functor
\[
An:\Sm/k\to \Top
\]
sending a smooth $k$-scheme $X$ to $X^\an$.  Taking the left Kan extension and geometric realization defines the functor $An^*: \Spc(k)\to \Top$ with right adjoint $An_*:\Top\to \Spc(k)$ sending a topological space $T$ to the constant presheaf on the singular complex $n\mapsto Maps(\Delta^n, T)$ of $T$. We have a similar adjoint pair in the pointed setting.  In particular, one has for $X$ a finite type $k$-scheme the natural transformation
\begin{equation}\label{eqn:Nat}
\epsilon_X:An^*(X)\to X^\an.
\end{equation}

 Panin-Pimenov-R\"ondigs \cite[theorem A.3.2]{PPR} define a closed simplicial symmetric monoidal model structure $\sM^{cm}_\bullet(k)$ on $\Spc_\bullet(k)$ for which $(An^*, An_*)$ becomes a Quillen pair of adjoint functors, with $An^*$ in addition a symmetric monoidal functor.  The identity functor $\sM^{cm}_\bullet(k)\to \sM_\bullet(k)$ is a left Quillen equivalence, as the weak equivalences are the same and the cofibrations in $\sM^{cm}_\bullet(k)$ are all cofibrations in $\sM_\bullet(k)$.
 
Let $\P^1_*$ denote $\P^1$ pointed by $1$. Let $j:\G_m\to \A^1$ be the inclusion (we give both schemes the base-point 1) and let $M(j)$ be the pointed mapping cone of $j$, with morphisms $\beta:M(j)\to \A^1/\G_m\to \P^1/\A^1$, and $\gamma:M(j)\to S^1\wedge \G_m=T$. Let $\Spt_{\P^1}(k)$ be the category of $\P^1_*$-spectra in $\Spc_\bullet(k)$ and define $\Spt_{\P^1/\A^1}(k)$, $\Spt_{M(j)}(k)$ similarly.  We  define model structures on $\Spt_{\P^1/\A^1}(k)$,  $\Spt_{\P^1}(k)$ and $\Spt_{M(j)}(k)$ using word for word the definition we have used for $\Spt_T(k)$.

The diagram of maps in $\Spc_\bullet(k)$
\[
\P^1_*\xrightarrow{\alpha}\P^1/\A^1\xleftarrow{\beta}M(j)\xrightarrow{\gamma}T
\]
give rise to functors
\begin{equation}\label{eqn:zigzag}
\Spt_{\P^1}(k)\xleftarrow{\alpha^*}\Spt_{\P^1/\A^1}(k)\xrightarrow{\beta^*}\Spt_{M(j)}(k)\xleftarrow{\gamma^*}\Spt_T(k)
\end{equation}
which induce equivalences on the respective homotopy categories \cite[proposition 2.13]{Jardine2}. The maps $\beta^*$ for instance is defined by sending a $\P^1/\A^1$ spectrum $\sE:=(E_0, E_1,\ldots), \epsilon_n:E_n\wedge \P^1/\A^1\to E_{n+1}$ to the $M(j)$-spectrum $(E_0, E_1,\ldots)$ with bonding maps 
$\beta^*(\epsilon_n):=\epsilon_n\circ(\id\wedge \beta)$.

Next, Panin-Pimenov-R\"ondigs define a model structure $\sM\sS^{cm}(k)$ on  $\Spt_{\P^1}(k)$  and extend $An^*$ to a left Quillen functor from $\sM\sS^{cm}(k)$ to the category of $S^2$ spectra (in compactly generated topological spaces), $\Spt_{S^2}$, using the fact that $An^*$ is symmetric monoidal and that $\P^1(\C)$ is homeomorphic to $S^2$. Denote this left Quillen functor by
\[
An^*_{\P^1}:\sM\sS^{cm}(k)\to \Spt_{S^2}.
\]
Explicitly, $An^*_{\P^1}((E_0, E_1,\ldots), \epsilon_*)$ is the $S^2$-spectrum $(An^*E_0, An^*E_1,\ldots)$ with bonding maps given by
\[
An^*E_n\wedge S^2\cong An^*E_n\wedge An^*(\P^1_*) \cong An^*(E_n\wedge \P^1_*)\xrightarrow{An^*(\epsilon_n)}An^*E_{n+1}.
\]
Let $LAn^*_{\P^1}:\Ho\sM\sS^{cm}(k)\to \Ho\Spt_{S^2}$ be left derived functor of $An^*_{\P^1}$. 

The Jardine model structure $\sM\sS_{\P^1}(k)$ described in \S\ref{sec:Cellular}  is different from  $\sM\sS^{cm}(k)$, but just as in the unstable setting, the identity functor $\sM\sS^{cm}(k)\to \sM\sS_{\P^1}(k)$ is a left Quillen equivalence  (see \cite[theorem A.5.6]{PPR}).

The homotopy category $\Ho \Spt_{S^2}$ is equivalent to the usual stable homotopy category of $S^1$-spectra in $\Spc_\bullet$, $\SH$, so putting all this together, the functor $LAn^*_{\P^1}$ induces the desired Betti realization functor, $
\Re_B^\sigma:\SH(k)\to \SH$.

The following identities are easy to show and are left to the reader.
\begin{equation}\label{eqn:SuspIdent}
\Re_B^\sigma(\Sigma_T\sE)\cong \Sigma^2\Re_B^\sigma(\sE);\ \Re_B^\sigma(\Sigma_{\G_m}\sE)\cong \Sigma\Re_B^\sigma(\sE).  
\end{equation}
In addition, we have: 
\begin{lem} \label{lem:CompBetti1}For $X\in \Sm/k$, $n\ge0$, the map $LAn^*(\Sigma^n_{\P^1}X_+)\to \Sigma^n_{S^2}X^\an_+$ induced by \eqref{eqn:Nat}   is an isomorphism in  $\Ho\Top_\bullet$.
\end{lem}

\begin{proof}  $X_+$ and $\P^1_*$ are cofibrant objects of $\sM^{cm}(k)$ \cite[Lemma A.10]{PPR}; as $\sM^{cm}(k)$ is a symmetric monoidal model category, 
$\Sigma^n_{\P^1}X_+$ is cofibrant in $\sM^{cm}(k)$. Thus $LAn^*(\Sigma^n_{\P^1}X_+)\to An^*\Sigma^n_{\P^1}X_+$ is an isomorphism. Since $An^*$ is a symmetric monoidal functor, $An^*\Sigma^n_{\P^1}X_+\cong \Sigma^n_{S^2}X^\an_+$.
\end{proof}

Here is our main theorem on the connectedness of the Betti realization of the slice tower.

\begin{thm}\label{thm:Connected} Suppose that $k$ has an embedding  $\sigma:k\hookrightarrow \C$. Let $\Re_B^\sigma:\SH(k)\to \SH$ be the associated Betti realization functor. For $\sE\in \SH(k)$, if  $\sE$ is topologically $N-1$-connected, then $\Re_B^\sigma(f^t_q\sE)$ is $q+N-1$ connected for all $q\in\Z$.
\end{thm}

\begin{proof}  Let $\sX=\Sigma^m_{S^1}\Sigma^n_{\G_m}\Sigma_T^\infty X_+$. Since $\Sigma_{\P^1}^\infty X_+$ is cofibrant in $\sM\sS^{cm}(k)$, the identities \eqref{eqn:SuspIdent}  together with lemma~\ref{lem:CompBetti1} show that  $\Re^\sigma_B(\sX)\cong\Sigma^{m+n}\Sigma^\infty X^\an_+$, and thus $\Re_B(\sX)$ is $m+n-1$ connected. As we have noted at the beginning of \S\ref{sec:Postnikov2}, $\tau^{N,q}\SH(k)$ is the smallest full subcategory of $\SH(k)$  containing the objects 
$\Sigma^m_{S^1}\Sigma^n_{\G_m}\Sigma^\infty_TX_+$, $X\in\Sm/k$, $m\ge N, n\ge q$, and closed under  small coproducts and taking cones of morphisms. As $An^*_{\P^1}$ is a left Quillen functor, $\Re_B^\sigma$ is a left adjoint and hence is compatible with arbitrary small coproducts; it is of course exact. As the property of being $q+N-1$ connected is similarly closed under small coproducts and taking cones of morphisms in $\SH$, it follows that $\Re_B^\sigma(\sX)$ is $q+N-1$ connected for all $\sX$ in $\tau^{N,q}\SH(k)$. . Thus, for each $\sE\in \SH(k)$,  $\Re_B^\sigma(f_{N,q}\sE)$ is $N+q-1$ connected. But by proposition~\ref{prop:SliceConn}(2), $f_{N,q}\sE\to f^t_q\sE$ is an isomorphism for $\sE$ topologically $N-1$-connected, completing the proof.
\end{proof}

Our next task is to say something about the Betti realization of symmetric products; for material on symmetric products, we refer the reader to the appendix~\ref{sec;sym}.

\begin{lem}\label{lem:CompBetti2} Let $X$ be a finite type $k$-scheme. Then for all $n\ge1$, the map  $LAn^*(\Sigma^n_{\P^1}X_+)\to \Sigma^n_{S^2} X^\an_+$ induced by \eqref{eqn:Nat}  is an isomorphism in $\Ho\Top_\bullet$.
\end{lem}

\begin{proof} Let $X_\bullet\to X$ be a  cdh hypercover with each $X_n$ smooth over $k$; such $X_\bullet$ exists since $k$ admits resolution of singularities.  By Voevodsky's theorem \cite[Theorem 4.2]{VoevNisCdh} comparing the unstable motivic homotopy categories for the Nisnevich and $cdh$ topologies,  for $n\ge1$, 
\[
\Sigma^n_{\P^1}X_{\bullet+}\cong \Sigma^n_{\P^1} X_+
\]
in $\sH_\bullet(k)$,  and hence we have an isomorphism after applying $LAn^*$ (actually, Voevodsky shows the finer result that $\Sigma_{S^1} X_{\bullet+}\to \Sigma_{S^1} X_+$ is an isomorphism).

As noted in the proof of lemma~\ref{lem:CompBetti1},  each term $\Sigma^n_{\P^1}X_{p+}$ in the simplicial object $\Sigma^n_{\P^1}X_\bullet$ is cofibrant, hence the canonical map $LAn^*\Sigma^n_{\P^1}X_{\bullet+}\to An^* \Sigma^n_{\P^1}X_{\bullet+}$
is a weak equivalence in $\Top_\bullet$. Thus we have the isomorphisms in $\Ho\Top_\bullet$
\[
LAn^*(\Sigma^n_{\P^1} X_+)\cong LAn^*(\Sigma^n_{\P^1}X_{\bullet+})
 \cong An^* \Sigma^n_{\P^1}X_{\bullet+}\cong |\Sigma^n_{S^2}X_{\bullet+}^\an|,
\]
where $|S|$ denotes the geometric realization of a simplicial space $S$. 

Since $X_\bullet\to X$ is a cdh hypercover, it follows that $X_\bullet^\an\to X^\an$ is a hypercover for the classical topology. In particular, the map
$|\Sigma^n_{S^2}X^\an_{\bullet+}|\to \Sigma^n_{S^2} X^\an_+$
is a homology isomorphism. As $n\ge1$, both spaces are simply connected, it therefore follows from the relative Hurewicz theorem that this map is an isomorphism in $\Ho\Top_\bullet$, which completes the proof.
\end{proof}

For a topological space $T$, we let $\Sym^nT:=T^n/\Sigma_n$, where the symmetric group $\Sigma_n$ acts by permuting the factors. If $T$ is pointed by $t\in T$, we write  $\Sym^n_\bullet T$ for  the pointed space $(\Sym^nT, \Sym^nt)$.

\begin{lem}\label{lem:CompBetti3}  For $Y\in\Sm/k$, $n\ge1$, $m, N\ge0$ integers, the natural map 
\[
LAn^*\Sigma_{\P^1}^n\Sym^N_\bullet\Sigma^m_{\P^1}Y_+\to \Sigma^n_{S^2}\Sym^N_\bullet \Sigma^m_{S^2}Y^\an_+
\]
is a weak equivalence in $\Top_\bullet$.
\end{lem}

\begin{proof}  We proceed by induction on $N\ge2$, the case $N=0$ being obvious and the case $N=1$ following from lemma~\ref{lem:CompBetti2}. We use the notation from appendix~\ref{sec;sym}.
 
We apply lemma~\ref{lem:SymInd},  taking $X=(\P^1)^m\times Y_+$, $A= (\P^1,1)^m\times Y_+$, giving the co-cartesian diagram 
\[
\xymatrix{
\Sym^N_\bullet(X,A) \ar[r]^{i}\ar[d]_{\bar\pi_N}&\Sym_\bullet^NX\ar[d]^{\Sym^N\pi }\\
\Sym^{N-1}_\bullet \Sigma^m_{\P^1}Y_+ \ar[r]_{st_N}&\Sym^N_\bullet \Sigma^m_{\P^1}Y_+}
\]
in $\Spc^\sC_\bullet(k)$, which we may further suspend to give the co-cartesian diagram in $\Spc^\sC_\bullet(k)$
\begin{equation}\label{eqn:SymQuot}
\xymatrix{
\Sigma^n_{\P^1}\Sym^N_\bullet(X,A) \ar[r]^{i}\ar[d]_{\bar\pi_N}&\Sigma^n_{\P^1}\Sym_\bullet^NX\ar[d]^{\Sym^N\pi }\\
\Sigma^n_{\P^1}\Sym^{N-1}_\bullet \Sigma^m_{\P^1}Y_+ \ar[r]_{st_N}&\Sigma^n_{\P^1}\Sym^N_\bullet \Sigma^m_{\P^1}Y_+.}
\end{equation}
Since the restriction functor $\Spc^\sC_\bullet(k)\to \Spc_\bullet(k)$ preserves co-cartesian diagrams, we may consider this last diagram as a co-cartesian diagram in $\Spc_\bullet(k)$.

The map $i$ is a monomorphism hence $i$ is a cofibration in the Jardine model structure on $\Spc_\bullet(k)$. Thus, \eqref{eqn:SymQuot}  is homotopy co-cartesian; as the cofibrations in $\sM^{cm}_\bullet(k)$ are cofibrations in the Jardine model structure, \eqref{eqn:SymQuot} is homotopy co-cartesian in  $\sM^{cm}_\bullet(k)$.

If we apply $(-)^\an$ to the diagram \eqref{eqn:SymQuot} we have the co-cartesian and homotopy co-cartesian diagram of cofibrant objects in $\Top_\bullet$
\begin{equation}\label{eqn:SymQuotAn}
\xymatrix{
\Sigma_{S^2}^n  \Sym^N_\bullet(X,A)^\an\ar[r]^i\ar[d] & \Sigma_{S^2}^n\Sym_\bullet^NX^\an\ar[d]^\pi\\
\Sigma^n_{S^2}\Sym^{N-1}_\bullet \Sigma^m_{S^2}Y^\an_+ \ar[r] &\Sigma_{S^2}^n\Sym^N_\bullet\Sigma^m_{S^2}Y^\an_+.}
\end{equation}

Since $LAn^*$ is the left derived functor of a left Quillen functor,  $LAn^*$ transforms  homotopy  co-cartesian diagrams  in  $\sM^{cm}_\bullet(k)$  into homotopy  co-cartesian diagrams in $\Top_\bullet$; in particular $LAn^*\hbox{\eqref{eqn:SymQuot}}$ is homotopy co-cartesian.  The map
\[
LAn^*\Sigma^n_{\P^1}\Sym^{N-1}_\bullet \Sigma^m_{\P^1}Y_+ \to \Sigma^n_{S^2}\Sym^{N-1}_\bullet \Sigma^m_{S^2}Y^\an_+
\]
is a weak equivalence by induction. As $\Sym^N_\bullet(X,A)$ is represented by a closed subscheme of $\Sym^N_\bullet X$,   the maps 
\[
LAn^* \Sigma_{\P^1}^n\Sym^N_\bullet X\to  \Sigma_{S^2}^n\Sym^N_\bullet X^\an,\ LAn^*\Sigma_{\P^1}^n\Sym^N_\bullet(X,A)\to \Sigma_{S^2}^n \Sym^N_\bullet(X,A)^\an
\]
are weak equivalences by lemma~\ref{lem:CompBetti2}, hence 
\[
LAn^*\Sigma_{\P^1}^n\Sym^{N}_\bullet\Sigma^m_{\P^1}Y_+\to
\Sigma_{S^2}^n\Sym^{N}_\bullet\Sigma^m_{S_2}Y^\an_+
\]
is a weak equivalence as well.
\end{proof}

We have the $\P^1$-spectrum
\[
(\Sigma^\infty_{\P^1}X_+)^\tr_\eff:=(\Sym^\infty_\bullet X_+, \Sym^\infty_\bullet  \Sigma_{\P^1}X_+, \ldots, \Sym^\infty_\bullet  \Sigma_{\P^1}^n X_+,\ldots).
\]
The bonding maps are defined via the product maps
\[
\Sym^N_\bullet   \Sigma_{\P^1}^n X_+\wedge \P^1_*=\Sym^N_\bullet   \Sigma_{\P^1}^n X_+\wedge \Sym^1_\bullet  \P^1_*\to \Sym^N_\bullet ( \Sigma_{\P^1}^n X_+\wedge \P^1_*).
\]
We let $(\Sigma^\infty_TX_+)^\tr_\eff\in\SH(k)$ be the object   corresponding to $(\Sigma^\infty_{\P^1}X_+)^\tr_\eff$ via the equivalence of categories $\SH(k)\sim \Ho\Spt_{\P^1}(k)$ described above.

For $S$ a pointed space, one has the  spectrum
\[
(\Sigma^\infty S)^\tr_\eff:=(\Sym^\infty_\bullet S, \Sym^\infty_\bullet \Sigma S, \ldots, \Sym^\infty_\bullet \Sigma^n S,\ldots).
\]
with bonding maps defined as above.  We let $H\Z:=(\Sigma^\infty S^0)^\tr_\eff$. The Dold-Thom theorem can be phased as

\begin{thm}[\hbox{\cite{DoldThom}}] Suppose $S$ has the homotopy type of a pointed countable CW complex. Then $\pi_n (\Sigma^\infty S)^\tr_\eff=\tilde{H}_n(S,\Z)$ for $n\in \Z$.
\end{thm}
In particular, $H\Z$ is isomorphic in $\SH$ to the Eilenberg-MacLane spectrum $EM(\Z)$. 

\begin{prop}\label{prop:BettiReal} For $X\in \Sm/k$, there is a natural isomorphism in $\SH$
\[
\Re_B^\sigma((\Sigma^\infty_TX_+)^\tr_\eff)\cong (\Sigma^\infty X^\an_+)^\tr_\eff
\]
\end{prop}

\begin{proof}  For a pointed space $S$, let  $(\Sigma_{S^2}^\infty S)^\tr_\eff$ be the $S^2$-spectrum 
\[
(\Sigma_{S^2}^\infty S)^\tr_\eff:=(\Sym^\infty S, \Sym^\infty \Sigma_{S^2} S, \ldots, \Sym^\infty \Sigma^n_{S^2} S,\ldots).
\]
 
We show that for  $X\in \Sm/k$, there is a natural isomorphism in $\Ho\,\Spt_{S^2}$
\[
LAn^*((\Sigma^\infty_{\P^1}X_+)^\tr_\eff)\cong (\Sigma^\infty_{S^2} X^\an_+)^\tr_\eff,
\]
from which the result follows directly. 

Up to isomorphism in $\Ho\,\Spt_{\P^1}(k)$, we can represent $(\Sigma^\infty_{\P^1}X_+)^\tr_\eff$   by the spectrum
\[
(\pt,  \Sym^1_\bullet \Sigma_{\P^1} X_+,\ldots,\Sym^{n}_\bullet \Sigma^n_{\P^1} X_+,\ldots)
\]
using the same formula as above for the bonding maps, followed by the evident stabilization map. We have a similar representation for 
$(\Sigma^\infty_{S^2} X^\an_+)^\tr_\eff$.

We have the natural map in $\Spt_{S^2}$
\begin{multline*}
LAn^*_{\P^1}(\pt,  \Sym^1_\bullet\Sigma_{\P^1} X_+,\ldots,\Sym^{n}_\bullet \Sigma^n_{\P^1} X_+,\ldots)\\
\xrightarrow{\phi}
(\pt, \Sym^1_\bullet\Sigma_{S^2}X^\an_+,\ldots, \Sym^n_\bullet\Sigma^n_{S^2}X^\an_+,\ldots)
\end{multline*}
which on the $m$th term in the sequence is equivalent in $\Ho\Top_\bullet$ to the natural map 
$\phi_m:LAn^*\Sym^m_\bullet\Sigma^m_{\P^1} X_+\to \Sym^m_\bullet\Sigma^m_{S^2}X^\an_+$. By lemma~\ref{lem:CompBetti3}, $\Sigma_{S^2}\phi_m$ is a weak equivalence in $\Top_\bullet$ for all $m$, hence $\phi$ is a stable weak equivalence.    
\end{proof}

R\"ondigs-{\O}stv{\ae}r \cite{RoendigsOstvar} consider the category $\ChSS^\tr_{\P^1}$ of symmetric $\Z^{tr}(\P^1_*)$-spectra in (unbounded) chain complexes of presheaves with transfer on $\Sm/k$. They define a model structure on $\ChSS^\tr_{\P^1}$;  the homotopy category $\Ho\ChSS^\tr_{\P^1}$ is denoted $\DM(k)$.  On has as well the category of (unbounded) effective motives, $\DM^\eff(k)$, defined as the localization of the unbounded derived category of Nisnevich sheaves with transfers with respect to the localizing category generated by objects $\Z^\tr(X\times \A^1)\to \Z^\tr(X)$. $\DM^\eff(k)$ contains Voevodsky's category $\DM_-^\eff(k)$ as a full triangulated subcategory, which in turn contains Voevodsky's category of effective geometric motives $\DM^\eff_\gm(k)$ as a full triangulated subcategory. One has the Spanier-Whitehead category  of geometric motives
$\DM^\eff_\gm(k)[-\otimes\Z(1)^{-1}]=:\DM_\gm(k)$; by Voevodsky's cancellation theorem \cite{VoevCanc}, the functors in the diagram of triangulated tensor categories
\[
\xymatrix{
\DM^\eff_\gm\ar[r]\ar[d]&\DM^\eff_-(k)\ar[r]&\DM^\eff(k)\ar[d]\\
\DM_\gm(k)\ar[rr]&&\DM(k)
}
\]
are all fully faithful embeddings. For details we refer the reader to  \cite[section 2.3]{RoendigsOstvar}.

We have the category of symmetric $T$-spectra $\Spt_{T}^\Sigma(k)$ (see \cite[\S4]{Jardine2}); giving this the model structure defined by Jardine {\it loc.\,cit.}\,defines the model category of symmetric motivic  spectra. We recall \cite[theorem 4.31]{Jardine2} which states that forgetting the symmetric structure induces an equivalence of triangulated categories $\Ho\Spt_T^\Sigma(k)\to\SH(k)$. 

R\"ondigs-{\O}stv{\ae}r \cite{RoendigsOstvar} define  a commutative monoid object $M\Z$ in $\Spt_{T}^\Sigma(k)$. One may consider the category $\Mod\text{-}M\Z$ of modules  for $M\Z$ in symmetric motivic spectra . Let $F:\Mod\text{-}M\Z\to \Spt^\Sigma_T(k)$ be the forgetful functor. Defining a morphism $f$ in $\Mod\text{-}M\Z$ to be a fibration, resp. weak equivalence, if $F(f)$ is so in $ \Spt^\Sigma_T(k)$  gives  $\Mod\text{-}M\Z$ a model category structure for which $F$ becomes a right Quillen functor, with left adjoint the free $M\Z$-module functor $\sE\mapsto M\Z\wedge\sE$. This yields the  functor $RF:\Ho\,\Mod\text{-}M\Z\to \SH(k)$.

The main result of \cite{RoendigsOstvar} is

\begin{thm}[\hbox{\cite[theorem 1.1]{RoendigsOstvar}}] \label{thm:RO} Let $k$ be a field of characteristic zero. Then there is an equivalence of $\DM(k)$ with $\Ho\,\Mod\text{-}M\Z$ as triangulated tensor categories,  sending $\Z^\tr(X)$ to 
$M\Z\wedge X_+$ for $X\in \Sm/k$.
\end{thm}

We write 
\[
\EM:\DM(k)\to \SH(k)
\]
for the  functor induced by the equivalence $\Ho\,\Mod\text{-}M\Z\cong \DM(k)$ and the forgetful functor $RF:\Ho\,\Mod\text{-}M\Z\to \SH(k)$, and $\Z^{tr}:\SH(k)\to \DM(k)$ for the left adjoint to $\EM$.

\begin{rem} The functor $\EM$ preserves arbitrary coproducts: this is a general fact about an exact functor $R$ between compactly generated triangulated categories admitting arbitrary coproducts, such that $R$ admits a left adjoint $L$ with the property that $L(A)$ is compact if $A$ is compact. In this case, $L$ is the functor $\Z^{tr}$, $\SH(k)$ has the compact generators 
$\Sigma^a_{S^1}\Sigma^b_{\G_m}\Sigma^\infty_TU_+$, $a,b\in\Z$, $U\in \Sm/k$, and $\Z^{tr}(\Sigma^a_{S^1}\Sigma^b_{\G_m}\Sigma^\infty_TU_+)=\Z^{tr}(U)(b)[a+b]$, which is compact in $\DM(k)$.
\end{rem}

\begin{lem}\label{lem:EM} For $X$ in $\Sm/k$,  $\EM(\Z^{tr}(X))\cong (\Sigma^\infty_T\Spec X_+)^\tr_\eff$    in $\SH(k)$.
\end{lem}

\begin{proof}  By theorem~\ref{thm:RO}, we need to show that $M\Z\wedge X_+$ and $(\Sigma^\infty_T\Spec X_+)^\tr_\eff$  are isomorphic  in $\SH(k)$.    R\"ondigs-{\O}stv{\ae}r \cite{RoendigsOstvar} construct $M\Z$ as the symmetric $T$-spectrum corresponding to the motivic functor $M\Z$ which sends $X\in \Sm/k$ to the Nisnevich sheaf of sets $\Z^{tr}(X)$ on $\Sm/k$, $\Z^{tr}(X)(Y)=\Cor_k(Y,X)$.  The structure of $M\Z$ as a motivic functor\footnote{See \cite{DRO} for details on motivic functors.} induces maps $A\wedge M\Z(B)\to M\Z(A\wedge B)$ for all finitely presented objects $A, B$ in $\Spc_\bullet(k)$, which allows one to the construct  the symmetric $T$-spectrum $M\Z$ from the motivic functor $M\Z$ in the evident manner. One can similarly construct a $\P^1$-spectrum $M\Z_{\P^1}$ as $(M\Z(S^0_k), M\Z(\P^1_*),\ldots, M\Z(\Sigma^n_{\P^1}S^0_k),\ldots)$, with bonding maps given using the maps mentioned above. Via the zig-zag diagram 
\eqref{eqn:zigzag}, we see that $M\Z_{\P^1}$ and $M\Z$ have isomorphic images in $\SH(k)$. Thus, it suffices to show that  $M\Z_{\P^1}\wedge X_+$ is isomorphic to  $(\Sigma^\infty_{\P^1}\Spec X_+)^\tr_\eff$ in $\Ho\,\Spt_{\P^1}(k)$. 

By lemma~\ref{lem:SymTransfer},  $(\Sigma^\infty_{\P^1}\Spec X_+)^\tr_\eff=(\Z^{tr}_\eff(X), \ldots, \Z^{tr}_\eff(\Sigma^n_{\P^1}X_+),\ldots)$.  Define $(\Sigma^\infty_{\P^1}\Spec X_+)^{\tr}$ as the $\P^1$-spectrum $(\Z^{tr}(X), \ldots, \Z^{tr}(\Sigma^n_{\P^1}X_+),\ldots)$ with the evident bonding maps. The canonical natural transformation $\Z^{tr}_\eff(-)\to \Z^{tr}(-)$ defines a map of $\P^1$-spectra
$\phi:(\Sigma^\infty_{\P^1}X_+)^\tr_\eff\to (\Sigma^\infty_{\P^1} X_+)^\tr$. 
Using the fact that   $\Ho\,\Spt_{\P^1}(k)$ is an additive category,  it is not hard to see that $\phi$ is a stable $\A^1$-weak equivalence. 

In general,  we need to see that  $(\Sigma^\infty_{\P^1}\Spec X_+)^\tr\cong M\Z_{\P^1}\wedge X_+$ in $\Ho\,\Spt_{\P^1}(k)$. It suffices to have an isomorphism in $\Ho\,\Mod\text{-}M\Z_{\P^1}$ and then apply the forgetful functor to $\SH(k)$; a natural isomorphism   in $\Ho\,\Mod\text{-}M\Z_{\P^1}$ is given by applying \cite[theorem 4.2, corollary 5.3 and lemma 5.4]{RoendigsOstvar}, after replacing $T$ with $\P^1_*$.
\end{proof}

Putting proposition~\ref{prop:BettiReal} together with lemma~\ref{lem:EM} yields:
\begin{prop}\label{prop:RealEM}  For $X\in \Sm/k$, there is a natural isomorphism in $\SH$
\[
\Re_B^\sigma(\EM(\Z^{tr}(X)))\cong (\Sigma^\infty X^\an_+)^\tr_\eff.
\]
\end{prop}

For a triangulated category $\sT$, we let   $\sT_{tor}$ be the full subcategory of objects $\sE$ such that $\Hom_{\sT}(A, \sE)\otimes\Q=0$
for all compact objects $A$ in $\sT$.   For $X\in \sT$, $N>0$ an integer,  we let $X/N$ denote an object of $\sT$ that fits into a distinguished triangle $X\xrightarrow{n\cdot\id}X\to X/N\to X[1]$.

\begin{rem}\label{rem:TorGen} If $\sT$ has a set $\sS$ of compact generators, then the proof of \cite[lemma 4.3]{NSO} shows that  $\sE$ is in $\sT_{tor}$ if and only if $\Hom_{\sT}(A, \sE)\otimes\Q=0$ for all $A\in \sS$. 
Take $q\ge-\infty$. As the objects $\Sigma_{S^1}^a\Sigma_{\G_m}^b\Sigma^\infty_TX_+$, $a,b\in\Z$, $b\ge q$, $X\in\Sm/k$, form a set of compact generators for  $\tau^{-\infty,q}\SH(k)$, using the Gersten spectral sequence as in the proof of lemma~\ref{lem:HomotopySheaf} shows that $\sE$ is in  $\tau^{-\infty,q}\SH(k)_{tor}$  if and only if $\Pi_{a,b}(\sE)$ is a sheaf of torsion abelian groups for  all $a,b\in\Z$, $b\ge q$.

In addition, this shows that, for  $q\ge-\infty$, $Tor_q:=\{\Sigma_{S^1}^a\Sigma_{\G_m}^b\Sigma^\infty_TX_+/N\}$, $a,b\in \Z$, $b\ge q$, $N>1$, $X\in\Sm/k$,
is a set of compact  generators for  $\tau^{-\infty,q}\SH(k)_{tor}$.  By \cite[theorem 2.1(2)]{NeemanDuality},  $\tau^{-\infty,q}\SH(k)_{tor}$ is the smallest localizing subcategory of $\tau^{-\infty,q}\SH(k)$ containing $Tor_q$. Thus,  for $q'<q$,  the inclusion functor $\tau^{-\infty,q}\SH(k)\to \tau^{-\infty,q'}\SH(k)$ maps $\tau^{-\infty,q}\SH(k)_{tor}$ to $\tau^{-\infty,q'}\SH(k)_{tor}$; in particular, 
\[
\tau^{-\infty,q}\SH(k)_{tor}=\tau^{-\infty,q}\SH(k)\cap \SH(k)_{tor}.
\]

Replacing $\Sigma_{S^1}^a\Sigma_{\G_m}^b\Sigma^\infty_TX_+$ with $\Z^{tr}(X)(b)[a]$, analogous results hold in $\DM(k)$; the results described in this remark are discussed in somewhat more detail in \cite[appendix B]{LevineConv}.
\end{rem}

We require the following result, which is essentially a rephrasing of the theorem of Suslin-Voevodsky \cite[theorem 8.3]{SuslinVoev}.

\begin{cor}\label{cor:SV} Take $M\in \DM^\eff(k)_{tor}$ and let $\sigma:k\hookrightarrow \C$ be an embedding. Suppose $k$ is algebraically closed. Then the Betti realization induces an isomorphism
\[
\Re_{B*}^\sigma:\Pi_{n,0}\EM(M)(k)\to \pi_n(\Re_B^\sigma(\EM(M))).
\]
for all $n\in\Z$.
\end{cor}

\begin{proof} Let $\sC\subset  \DM^\eff(k)_{tor}$ be the full subcategory of objects $M$ for which the result holds. Then $\sC$ is a triangulated subcategory of $\DM^\eff(k)_{tor}$. We have already noted that the functor $\EM$ preserves arbitrary coproducts; since $\Re_{B*}^\sigma$ is a left adjoint, this functor preserves  arbitrary coproducts as well. Since both $\Pi_{n,0}(-)$ and $\pi_n(-)$ commute with arbitrary coproducts,  $\sC$ is closed under  arbitrary coproducts. 

By remark~\ref{rem:TorGen} and the fact that $\Z^{tr}(X)(b)$ is a summand of $\Z^{tr}(X\times\P^b)$ for $b\ge0$,  the objects $\Z^{tr}(X)[n]/N$,  $X\in\Sm/k$, $n\in\Z$, $N>1$, form a set of compact generators for  $\DM^\eff(k)_{tor}$, and it suffices to see that  all these objects  are in $\sC$. 

Let $M=\Z^{tr}(X)/N$. Then $\EM(M)\cong(\Sigma^\infty_TX_+)^\tr/N$ by the definition of $\EM$ and lemma~\ref{lem:EM}.  As $\EM$ has left adjoint $M\Z\wedge(-)$ (and similarly for $EM:D(\Ab)\to \SH$), we have natural isomorphisms
\begin{align}\label{align:SVIso}
&\Pi_{n,0}((\Sigma^\infty_TX_+)^\tr/N)(k)\cong H^{Sus}_n(X,\Z/N)\\
&\pi_n((\Sigma^\infty X^\an_+)^\tr/N)\cong H^{sing}_n(X^\an,\Z/N).\notag
\end{align}

Each element $\alpha\in H^{Sus}_n(X,\Z/N)$ is represented by a map of pairs of schemes $\tilde{\alpha}:(\Delta_k^n,\del\Delta_k^n)\to (\Sym^{MN}X,N\times (\Sym^{M}X))$, where
\[
N\times :\Sym^{M}X\to \Sym^{MN}X
\]
is the multiplication map, $\Delta^n_k:=\Spec k[t_0,\ldots, t_n]/\sum_it_i-1$ is the algebraic $n$-simplex over $k$ and $\del\Delta^n_k\subset \Delta^n_k$ is the closed subscheme defined by $\prod_{i=0}^nt_i=0$. From this representation, one sees that the map
\[
\Re_{B*}^\sigma:\Pi_{n,0}((\Sigma^\infty_TX_+)^\tr/N)\to \pi_n((\Sigma^\infty X^\an_+)^\tr/N)
\]
is compatible, via the isomorphisms \eqref{align:SVIso}, with the map
\[
H^{Sus}_n(X,\Z/N)\to H^{sing}_n(X^\an,\Z/N)
\]
sending $\tilde{\alpha}$ to $\tilde{\alpha}^\an\in \pi_n((\Sigma^\infty X^\an_+)^\tr/N)\cong  H^{sing}_n(X^\an,\Z/N)$.
By the Suslin-Voevodsky theorem \cite[theorem 8.3]{SuslinVoev}, this latter map induces an isomorphism
$H^{Sus}_n(X,\Z/N)\to H^{sing}_n(X^\an,\Z/N)$, which completes the proof.
\end{proof}

\section{Proof of the main theorem} \label{sec:Proof}  We begin by studying the layers and slices of suspension spectra in $\SH(k)$. For a field $F$, the {\em cohomological dimension} of $F$ is by definition the cohomological dimension of the absolute Galois group of $F$  \cite[I, \S 3.1]{Serre}.

We have the canonical distinguished triangle of endofunctors on $\SH(k)$
\[
f^t_{q+1}\to f^t_q\to s_q\to f^t_{q+1}[1].
\]
\begin{lem} \label{lem:RatlVanish} Suppose $k$ has  finite cohomological dimension. Then for $X\in \Sm/k$ of dimension $d$ over $k$,  and for $q\ge d+1$, $f^t_q(\Sigma^\infty_T X_+)$ and  $s_q(\Sigma^\infty_T X_+)$ are in $\SH(k)_{tor}$.
In particular, $f_q^t(\mS_k)$ and $s_q(\mS_k)$  are  in  $\SH(k)_{tor}$ for $q\ge1$.
\end{lem}

\begin{proof} By remark~\ref{rem:TorGen}, it suffices to show that the homotopy sheaves $\Pi_{a,b}f^t_q(\Sigma^\infty_T X_+)$ are torsion for $a\in\Z$, $b\ge q\ge d+1$. For $a, b$ in this range, the universal property of $f^t_q(\Sigma^\infty_T X_+)\to \Sigma^\infty_T X_+$ gives us an isomorphism $\Pi_{a,b}f^t_q(\Sigma^\infty_T X_+)\to \Pi_{a,b}\Sigma^\infty_T X_+$, 
so it suffices to see that $\Pi_{a,b}\Sigma^\infty_T X_+$ is a torsion sheaf for $b\ge d+1$, $a\in \Z$. 

If $y\in Y\in \Sm/k$ is a point, then since the $\Pi_{a,b}\Sigma^\infty_T X_+$ are strictly $\A^1$-invariant sheaves \cite[corollary 6.2.9]{MorelConn}, the restriction $\Pi_{a,b}\Sigma^\infty_T X_+(\sO_{Y,y})\to \Pi_{a,b}\Sigma^\infty_T X_+(k(Y))$
 is injective (by \cite[lemma 3.3.4]{MorelLec}), so it suffices to see that $\Pi_{a,b}\Sigma^\infty_T X_+(F)$ is torsion for all fields finitely generated over $k$, $a\in\Z$ and $b\ge d+1$. This is  \cite[proposition 6.9(2)]{LevineConv}.
\end{proof}

We recall that, by   Pelaez's theorem, $s_q(\mS_k)$ is the motivic Eilenberg-MacLane spectrum of a motive $\pi^\mu_q(\mS_k)(q)[2q]$:
\[
s_q(\mS_k)\cong \EM(\pi^\mu_q(\mS_k)(q)[2q]).
\]
We also know that $\pi^\mu_q(\mS_k)(q)[2q]$ is  in $\Z(q)\otimes\DM^\eff(k)$, hence  
$\pi^\mu_q(\mS_k)$ is in $\DM^\eff(k)$. 

\begin{lem}\label{lem:Torsion} Suppose $k$ has  finite cohomological dimension. Then for $X\in \Sm/k$ of dimension $d$ over $k$,  $\pi^\mu_q(\Sigma^\infty_T X_+)$ is in $\DM^\eff(k)_{tor}$   for $q>d$. In particular, $\pi^\mu_q(\mS_k)$ is  in $\DM^\eff(k)_{tor}$  for $q>0$.
\end{lem}

\begin{proof} Take $q>d$.  By lemma~\ref{lem:RatlVanish}, $s_q(\Sigma^\infty_T X_+)$ is in $\SH(k)_{tor}$. As  
\[
\Hom_{\DM(k)}(\Z^{tr}(Y)(b)[a],\pi^\mu_q(\Sigma^\infty_T X_+))=\Hom_{\SH(k)}(\Sigma^{a+b-d}_{S^1}\Sigma_{\G_m}^{b-d}\Sigma_T^\infty Y_+, s_q(\Sigma^\infty_T X_+)),
\]
it follows that  $\pi^\mu_q(\Sigma^\infty_T X_+))$ is in $\DM(k)_{tor}$. Since $q\ge0$, $\pi^\mu_q(\Sigma^\infty_T X_+))$ is in $\DM^\eff(k)$, and following remark~\ref{rem:TorGen}, $\DM^\eff(k)_{tor}=\DM^\eff(k)\cap \DM(k)_{tor}$.
\end{proof}

\begin{lem} \label{lem:TorComp} If $\sE$ is in $\SH(k)_{tor}$, then for all $q$,  $f_q^t(\sE)$ and $s_q(\sE)$ are in $\SH(k)_{tor}$, and $\pi^\mu_q(\sE)$ is in $\DM(k)_{tor}$.
\end{lem}

\begin{proof} For $\pi^\mu_q(\sE)$, it suffices to see that $\Hom_{\DM(k)}(\Z^{tr}(X)(b)[a], \pi^\mu_q(\sE))$ is torsion for all $a,b\in\Z$, $X\in\Sm/k$. Via the adjoint property of $\EM$, this is the same as checking that $s_q(\sE)$ is in $\SH(k)_{tor}$. Using the fact that $\SH(k)_{tor}$ is triangulated, it suffices to check that $f_q^t(\sE)$  is in $\SH(k)_{tor}$ for all $q$. By remark~\ref{rem:TorGen}, it suffices to see that 
$f_q^t(\sE)$  is in $\tau^{-\infty,q}\SH(k)_{tor}$ and for this it suffices to see that the homotopy sheaves $\Pi_{a,b}(f_q^t(\sE))$ are torsion for all $a\in\Z$, $b\ge q$. As $\Pi_{a,b}(f_q^t(\sE))=\Pi_{a,b}(\sE)$ for $b\ge q$, this follows from our assumption that $\sE$ is in $\SH(k)_{tor}$, together with remark~\ref{rem:TorGen}.
\end{proof}
 
For the remainder of this section, we assume that $k$ is algebraically closed and admits an embedding $\sigma:k\hookrightarrow \C$. We write $\Re$ for $\Re_B^\sigma:\SH(k)\to \SH$. 

\begin{prop} \label{prop:SV}  The map
$\Re_*:\Pi_{n,0}(s_q(\mS_k))(k)\to \pi_n(\Re(s_q(\mS_k)))$
is an isomorphism for all $q$ and $n$.
\end{prop}

\begin{proof} We note that $\mS_k$ is in $\SH^\eff(k)=\tau^{-\infty,0}\SH(k)$, hence for $q<0$, $s_q(\mS_k)=0$. For $q=0$, $s_0(\mS_k)\cong M\Z$ by Voevodsky's theorem \cite{VoevS0}, hence $\Re(s_0(\mS_k))=H\Z$ by proposition~\ref{prop:RealEM}. We thus have 
\[
\Pi_{n,0}s_0(\mS_k)(k)=\Pi_{n,0}M\Z(k)=H^{Sus}_n(\Spec k, \Z)=\begin{cases}0&\text{ for }n\neq0\\\Z&
\text{ for }n=0.\end{cases}
\]
Similarly, 
\[
\pi_n H\Z=\begin{cases}0&\text{ for }n\neq0\\\Z
&\text{ for }n=0.\end{cases}
\]
The unit in $\Pi_{0,0}M\Z(k)=H_0^{Sus}(\Spec k,\Z)=\Z$ is induced by the  unit map $S^0_k=\Sym^1_\bullet S^0_k\to \Sym^\infty_\bullet S^0_k=M\Z_0$ which  goes over to the unit in $H\Z=An_{\P^1}^*M\Z$ under the Betti realization, and therefore the Betti realization induces an isomorphism $\Pi_{0,0}M\Z(k)\to \pi_0H\Z$. This handles the cases $q\le 0$.

For $q>0$, $\pi^\mu_q$ is in $\DM(k)^\eff_{tor}$ by  lemma~\ref{lem:Torsion}. As $s_q(\mS_k)=\EM(\pi^\mu_q(q)[2q])$,  corollary~\ref{cor:SV} shows that 
\[
\Re_*: \Pi_{n,0}(s_q(\mS_k))(k)\to  \pi_n(\Re_B(s_q(\mS_k)))
\]
is  an isomorphism for all $n\in\Z$.
\end{proof}

\begin{lem} \label{lem:FFReduction} Suppose that $Re$ induces an isomorphism
\[
\Re_*:\Pi_{n,0}(\mS_k)(k)\to \pi_n(\mS)
\]
for all $n$. Then  the constant presheaf functor $c:\SH\to \SH(k)$ is fully faithful.
\end{lem}

\begin{proof} We first recall the construction of $c:\SH\to \SH(k)$. 
Let $c:\Spc_\bullet\to\Spc_\bullet(c)$ be the constant presheaf functor; note that $c$ is a monoidal functor and is left adjoint to the functor $ev_k$, $\sX\mapsto \sX(k)$. Extend $c$ to $c_T:\Spt\to \Spt_T(k)$ by sending a spectrum $E=(E_0, E_1,\ldots)$ to the $T$-spectrum $(c(E_0), \Sigma_{\G_m} c(E_1),\ldots, \Sigma^n_{\G_m}c(E_n),\ldots)$ with bonding maps  given as the composition
\[
\Sigma^n_{\G_m}c(E_n)\wedge T=c(E_n)\wedge\G_m^{\wedge n}\wedge S^1\wedge\G_m\cong
\Sigma^{n+1}_{\G_m}c(E_n\wedge S^1)\xrightarrow{\Sigma^{n+1}_{\G_m}(c(\epsilon_{E,n}))}
\Sigma^{n+1}_{\G_m}c(E_{n+1}).
\]
$c_T$ is a left Quillen functor with right adjoint the functor sending a $T$-spectrum $\sE=(\sE_0,\sE_1,\ldots)$ to the spectrum $ev^T_k\sE:=(\sE_0(k), \Omega_{\G_m}(\sE_1)(k),\ldots, \Omega^n_{\G_m}(\sE_n)(k),\ldots)$; bonding maps are defined by applying $ev_k$ to 
\[
\Omega^n_{\G_m}(\sE_n)\wedge S^1\to \Omega^{n+1}_{\G_m}(\sE_n\wedge S^1\wedge \G_m)\xrightarrow{\Omega^{n+1}_{\G_m}(\epsilon_{\sE,n})}\Omega^{n+1}_{\G_m}(\sE_{n+1}).
\]
The functor $c:\SH\to \SH(k)$ is by definition $Lc_T$ and is thus exact and compatible with small coproducts.

Our hypothesis on $\Re_*$ can be expressed as saying that 
\[
\Re_*:[\Sigma^n_{S^1}\mS_k, \mS_k]_{\SH(k)}\to [\Re(\Sigma^n_{S^1}\mS_k),\Re(\mS_k)]_{\SH}
\]
is an isomorphism for all $n$; since $\Re_*\circ c\cong \id$ and $c(\mS)\cong\mS_k$, this shows that $c_*:[\Sigma^n\mS, \mS]_{\SH}\to [c(\Sigma^n\mS), c(\mS)]_{\SH(k)}$ is an isomorphism for all $n$.

Let $\sR\subset \SH$ be the full subcategory with objects $F$ such that $c_*:[\Sigma^n\mS, F]_{\SH}\to [c(\Sigma^n\mS), c(F)]_{\SH(k)}$ is an isomorphism for all $n$. Both $\Sigma^n\mS$ and $c(\Sigma^n\mS)\cong\Sigma^n_{S^1}\mS_k$ are compact and $c$ is compatible with small coproducts, hence $\sR$ is a localizing subcategory of $\SH$. As $\SH$ is generated as a localizing category by $\mS$, it follows that $\sR=\SH$. 

Now take $F\in \SH$ and let $\sL_F\subset \SH$ be the  full subcategory with objects $E$ such that $c_*:[E, F]_{\SH}\to [c(E), c(F)]_{\SH(k)}$ is an isomorphism. Clearly $\sL_F$ is a thick subcategory and is closed under small coproducts, hence is a localizing subcategory of $\SH$; as we have already seen that $\sL_F$ contains $\mS$, this shows that $\sL_F=\SH$, completing the proof.
\end{proof}

\begin{lem} \label{lem:FFReduction2} Let $k$ be an algebraically closed field of characteristic zero. Suppose that 
 the constant presheaf functor $c_L:\SH\to \SH(L)$ is fully faithful for all algebraically closed subfields of $k$ which have finite transcendence dimension over $\Q$.  Then  the constant presheaf functor $c:\SH\to \SH(k)$ is fully faithful.
\end{lem}

\begin{proof} Let $L\subset L'$ be subfields of $k$  and let $f_{L'/L}:\Spec L'\to \Spec L$ be the corresponding morphism of schemes. We have the diagram
\[
\xymatrix{
\SH\ar[r]^{c_{L'}}\ar[dr]_{c_L}&\SH(L')\\
&\SH(L),\ar[u]_{f_{L'/L}^*}}
\]
which is commutative up to natural isomorphism. In particular, $f_{L'/L}^*c_L(E)\cong c_{L'}(E)$ for each $E\in\SH$. 

Choose a set of algebraically closed subfields $L_\alpha$ of $k$ of finite transcendence dimension over $\Q$, indexed by a well-ordered set $A$, with $k=\cup_\alpha L_\alpha$. Take $E, F\in\SH$.  By \cite[proposition A.1.2]{AyoubRigid}, the  map
\[
\colim_\alpha \Hom_{\SH(L_\alpha)}(c_{L_\alpha}(E), c_{L_\alpha}(F))\to  \Hom_{\SH(k)}(c(E), c(F))
\]
induced by the system of functors $f_{k/L_\alpha}^*$ is an isomorphism, from which the lemma follows directly.
\end{proof}

Combining lemma~\ref{lem:FFReduction} and lemma~\ref{lem:FFReduction2},  our main theorem~\ref{IntroThm:Main} follows from 

\begin{thm} For $k$ algebraically closed of characteristic zero, with embedding $\sigma:k\hookrightarrow\C$, the map $\Re_*:\Pi_{n,0}(\mS_k)(k)\to \pi_n(\mS)$ is an isomorphism for all $n$. 
\end{thm}

\begin{proof} First we consider the case $n=0$.  By Morel's theorem \cite[lemma 3.10, corollary 6.41]{MorelA1}, $\Pi_{0,0}(\mS_k)(k)=\GW(k)$, which is isomorphic to $\Z$ via the dimension function, as $k$ is algebraically closed. This shows that the map $\mS_k\to s_0(\mS_k)\cong M\Z$ induces an isomorphism 
\[
\Pi_{0,0}(\mS_k)(k)\to \Pi_{0,0}M\Z(k)=\Z
\]
Similarly, the first Postnikov layer for $\mS$, $\mS\to H\Z$, arises from the isomorphism $\pi_0(\mS)\cong \Z$. 
This gives us the commutative diagram
\[
\xymatrix{
\Pi_{0,0}\mS_k(k)\ar[d]_{\Re_*}\ar@{=}[r]&\Pi_{0,0}s_0\mS_k(k)\ar[d]_{\Re_*}\ar@{=}[r]&\Z\ar@{=}[d]&\\
\pi_0(\mS) \ar@{=}[r]&\pi_0(H\Z)\ar@{=}[r]&\Z
}
\]
from which it follows that $\Re_*:\Pi_{0,0}(\mS_k)(k)\to \pi_0(\mS)$ is an isomorphism.

Next, consider the slice tower for $\mS_k$. We have the distinguished triangle
\[
f^t_1\mS_k\to \mS_k\to s_0\mS_k\to  f^t_1\mS_k[1]
\]
with $s_0\mS_k\cong M\Z$.

We have already seen that the map $\Pi_{0,0}\mS_k(k)\to \Pi_{0,0}M\Z(k)=\Z$
is an isomorphism. Using Morel's connectedness theorem \cite[theorem 4.2.10]{MorelLec} plus \cite[lemma 4.3.11]{MorelLec}, we see that $\mS_k$ is topologically -1 connected, hence $f^t_1\mS_k$ is also topologically -1 connected (proposition~\ref{prop:SliceConn}(1)). From the long exact sequence
\[
\ldots\to \Pi_{a+1,0}M\Z(k)\to \Pi_{a,0}f^t_1\mS_k(k)\to  \Pi_{a,0}\mS_k(k)\to \Pi_{a,0}M\Z(k)\to \ldots
\]
and the fact that  $\Pi_{a,0}M\Z(k)=H^{-a}(k,\Z(0))=0$ for $a\neq0$, we see that 
$\Pi_{a,0}f^t_1\mS_k(k)= \Pi_{a,0}\mS_k(k)$ for $a\neq0$ and  $\Pi_{0,0}f^t_1\mS_k(k)= 0$.
Finally, by proposition~\ref{prop:RealEM}, $\Re(M\Z)$ is the usual Eilenberg-MacLane spectrum $H\Z$, hence
\[
\pi_a(\Re(M\Z))=\begin{cases}0&\text{ for }a\neq 0\\\Z&\text{ for }a= 0.\end{cases}
\]
As $\Re$ is exact, it suffices to show that
\[
\Re_*:\Pi_{a,0}f^t_1\mS_k(k)\to \pi_{a}(Re(f^t_1\mS_k))
\]
is an isomorphism for all $a$.

For this we use the spectral sequences associated to the slice tower
\[
\ldots\to f^t_{n+1}\mS_k\to f^t_n\mS_k\to\ldots\to f^t_1\mS_k
\]
and its Betti realization
\[
\ldots\to \Re(f^t_{n+1}\mS_k)\to \Re(f^t_n\mS_k)\to\ldots\to \Re(f^t_1\mS_k).
\]
By \cite[theorem 4]{LevineConv}, the first tower gives a strongly convergent spectral sequence
\[
E^2_{p,q}(I)=\Pi_{p+q,0}(s_{q}\mS_k)(k)\Longrightarrow \Pi_{p+q,0}f^t_1\mS_k(k).
\]
By theorem~\ref{thm:Connected}, $Re(f^t_q(\mS_k))$ is $q-1$ connected, hence the Betti tower gives us the strongly convergent spectral sequence
\[
E^2_{p,q}(II)=\pi_{p+q}(\Re(s_{q}(\mS_k)))\Longrightarrow \pi_{p+q}\Re(f^t_1\mS_k).
\]
As $\Re_*$ gives a map of spectral sequences  $E(I)\to E(II)$,  it suffices to show that
\[
\Re_*:\Pi_{n,0}(s_{q}\mS_k)(k)\to \pi_n(\Re(s_{q}\mS_k))
\]
is an isomorphism for all $q>0$ and all $n$. This is proposition~\ref{prop:SV}.
\end{proof}

\section{The Suslin-Voevodsky theorem for homotopy}\label{sec:SVHpty}

\begin{thm} \label{thm:Embedding} Suppose $k$ is algebraically closed of characteristic zero, with an embedding $\sigma:k\hookrightarrow \C$. Then for $\sE\in\SH^\eff(k)_{tor}$, the map
\[
\Re_{B*}^\sigma:\Pi_{n,0}\sE(k)\to \pi_n(\Re_B^\sigma)
\]
is an isomorphism for all $n\in\Z$.
\end{thm}

\begin{proof} The exact functor $\Re_{B*}^\sigma$, the homotopy sheaves and homotopy groups are all compatible  with small coproducts. Thus, the full subcategory of $\SH^\eff(k)$ of objects $\sE$ for which the theorem holds is a localizing subcategory. Furthermore, by remark~\ref{rem:TorGen}, $\SH^\eff(k)_{tor}$ admits a set of compact generators,  namely, the suspension spectra $\Sigma^a_{S^1}\Sigma^b_{\G_m}\Sigma^\infty_T X_+/N$, $a\in\Z$, $b\ge0$, $N>1$,  $X\in\Sm/k$, so it suffices to prove the result for these generators. Letting $\SH_\fin(k)$ be the thick subcategory of $\SH(k)$ generated by the objects $\Sigma^a_{S^1}\Sigma^b_{\G_m}\Sigma^\infty_T X_+$, $a, b\in\Z$,  $X\in\Sm/k$, it suffices to prove the theorem  for $\sE\in \SH_\fin(k)\cap \SH^\eff(k)_{tor}$.

Since $\sE$ is in $\SH^\eff(k)$, we have $\sE=f^t_0\sE$. By \cite[theorem 4]{LevineConv}, as $\sE$ is in $\SH_\fin(k)$,  the tower
\[
\ldots\to f_{n+1}^t\sE\to  f_n^t\sE\to\ldots\to f_0\sE=\sE
\]
gives rise to a strongly convergent spectral sequence
\[
E^2_{p,q}=\Pi_{p+q,0}s_q\sE(k)\Longrightarrow \Pi_{p+q,0}\sE(k).
\]
Furthermore, by \cite[proposition 6.9(3)]{LevineConv} there is an integer $N$ such that $\sE$ is topologically $N-1$-connected. By theorem~\ref{thm:Connected}, $\Re_B^\sigma(f^t_n\sE)$ is $n+N-1$ connected for all $n\in\Z$, and hence the tower
\[
\ldots\to \Re_B^\sigma(f^t_{n+1}\sE) \to \Re_B^\sigma(f^t_n\sE)\to \ldots\to \Re_B^\sigma(f^t_{0}\sE)
=\Re_B^\sigma(\sE)
\]
defines a strongly convergent spectral sequence
\[
E^2_{p,q}=\pi_{p+q}\Re_B^\sigma(s_q\sE)\Longrightarrow \pi_{p+q}\Re_B^\sigma\sE.
\]
Since $f_0\sE=\sE$, it follows that $s_q\sE=0$ for $q<0$; as $\sE$ is in $\SH(k)_{tor}$, $\pi^\mu_q\sE$ is in $\DM^\eff(k)_{tor}=\DM^\eff(k)\cap\DM(k)_{tor}$ for $q\ge0$ by  lemma~\ref{lem:TorComp}. By corollary~\ref{cor:SV}, the map
\[
\Re_{B*}^\sigma:\Pi_{p+q,0}s_q\sE(k)\to \pi_{p+q}\Re_B^\sigma(s_q\sE)
\]
is an isomorphism for all $p,q$; as both spectral sequences are strongly convergent, 
\[
\Re_{B*}^\sigma:\Pi_{n,0}\sE(k)\to \pi_n(\Re_B^\sigma)
\]
is an isomorphism for all $n\in\Z$, as desired.
\end{proof}

As a special case,  we have the homotopy analog of the theorem of Suslin-Voevodsky promised in the introduction (theorem~\ref{IntroThm:Main2}):

\begin{cor} \label{cor:SVHomotopy} Let $k$ be an algebraically closed field of characteristic zero with an embedding $\sigma:k\hookrightarrow \C$. Then for all $X\in \Sm/k$, all integers $N>1$ and $n\in\Z$, the map
\[
\Re_{B*}^\sigma:\Pi_{n,0}(\Sigma^\infty_TX_+;\Z/N)(k)\to \pi_n(\Sigma^\infty X^\an_+;\Z/N)
\]
is an isomorphism. Here $\Pi_{n,0}(-;\Z/N)$ and $\pi_n(-;\Z/N)$ are the homotopy sheaves, resp. homotopy groups, with mod $N$ coefficients.
\end{cor}

\begin{proof} We note that $\Pi_{n,0}(\sE;\Z/N)$ is by definition $\Pi_{n,0}(\sE/N)$, and similarly for $\pi_n(E;\Z/N)$. We may apply theorem~\ref{thm:Embedding} to the object $\Sigma^\infty_TX_+/N$, which is in 
$ \SH(k)_{\fin}\cap \SH^\eff(k)_{tor}$; we need only note that, by lemma~\ref{lem:CompBetti1}, $\Re_B^\sigma(\Sigma^\infty_TX_+/N)\cong \Sigma^\infty X^\an_+/N$.
\end{proof}

\section{Slices of the sphere spectrum} \label{sec:VoevConj} Voevodsky has stated a conjecture \cite[conjecture 9]{VoevOpen} giving a formula for the slices of $\mS_k$ in terms of the Adams-Novikov spectral sequence for the homotopy groups of $\mS$. This conjecture follows from  properties of the motivic Thom spectrum $\MGL$, together with a result of Hopkins-Morel \cite{HopkinsMorel} on the slices of $\MGL$, now available through the preprint of M. Hoyois \cite{Hoyois}. We give some of the details of the proof of Voevodsky's conjecture, without any claim to originality.

We first recall Voevodsky's conjecture.

For a cosimplicial abelian group $p\mapsto A^p$, let $(A^*, d)$ be associated complex, with differential the usual alternating sum of the coface maps. We have as well the quasi-isomorphic normalized subcomplex $NA^*\hookrightarrow  A^*$, with $NA^p=\cap_{i=1}^p\ker s^p_i$, where $s^p_i:A^p\to A^{p-1}$ is the $i$th co-degeneracy map.

Consider the cosimplicial spectrum
\[\vbox{
$\xymatrixcolsep{10pt}
MU^{\wedge *}:=\cdots \ \xymatrix{\ar@<16pt>[r]\ar@<-16pt>[r]\lower2pt\hbox to0pt{\kern 7pt$\vdots$\hss}&\ar@<10pt>[l]\ar@<-10pt>[l]}\ MU^{\wedge n+1}\ \xymatrix{\ar@<16pt>[r]\ar@<-16pt>[r]\lower2pt\hbox to0pt{\kern 7pt$\vdots$\hss}&\ar@<10pt>[l]\ar@<-10pt>[l]}\ 
MU^{\wedge n}\ \xymatrix{\ar@<16pt>[r]\ar@<-16pt>[r]\lower2pt\hbox to0pt{\kern 7pt$\vdots$\hss}&\ar@<10pt>[l]\ar@<-10pt>[l]}\ \cdots\ \xymatrix{\ar@<12pt>[r]\ar@<-12pt>[r]\ar[r]&\ar@<6pt>[l]\ar@<-6pt>[l]}\  MU^{\wedge 2}\ 
\xymatrix{\ar[r]& \ar@<6pt>[l]\ar@<-6pt>[l]}\ MU
 \setlength{\baselineskip}{4\baselineskip}$}
\]
with $MU^{\wedge n}$ in degree $n-1$. The maps $\leftarrow$ insert the unit in the various factors, and the maps $\to$ are multiplication maps.

Applying $\pi_*$ and taking the usual alternating sum of the coface maps gives the complex of graded abelian groups (with $\pi_*(MU)$ in cohomological degree 0)
\[
\pi_*(MU^{\wedge *+1})=\pi_*(MU)\to \pi_*(MU\wedge MU)\to\ldots\to \pi_*(MU^{\wedge n})\to\ldots
\]

Let $p:MU\to \overline{MU}$ be the homotopy cofiber of the unit map $\mS\to MU$. We have the canonical isomorphism (of left $MU$-modules)
\[
MU\wedge MU\cong MU[b_1,b_2,\ldots]:=\vee_I MU\cdot b^I,
\]
where for a monomial $b^I$, $I=(i_1,\ldots, i_r)$, we take $MUb^I$ to mean $\Sigma^{2\sum_jj\cdot i_j}MU$. The unit map $MU\wedge \mS\to MU\wedge MU$ is   split by the multiplication and thus the map $MU\wedge MU^{\wedge n-1}\to MU\wedge\overline{MU}^{\wedge n-1}$ is canonically split.  Via this splitting, the subgroups $\pi_*(MU\wedge \overline{MU}^{\wedge n})$ form a  graded subcomplex of 
$\pi_*(MU^{\wedge *+1})$, which we denote by $\pi_*(NMU)^*$. This is in fact the normalized subcomplex $N\pi_*(MU^{\wedge *+1})$ of $\pi_*(MU^{\wedge *+1})$; in particular,  the inclusion $\alpha:\pi_*(NMU)^*\to \pi_*(MU^{\wedge *+1})$  is a quasi-isomorphism.

Furthermore, via this  split injection  $\pi_*(MU\wedge \overline{MU}^{\wedge n})$ is identified with an ideal in a polynomial algebra over the Lazard ring $\L=\pi_*(MU)$:
\[
\pi_*(MU\wedge \overline{MU}^{\wedge n}))=\L\otimes(\Z[b_1, b_2,\ldots]_+)^{\otimes n}
\]
where $\Z[-]_+$ means the ideal generated by all the variables $b_i$. The grading is given by setting $\deg b_m=2m$ and using the grading in $\L$ induced by the isomorphism $\pi_*(MU)\cong \L$. In particular, we have for each $q\ge0$ the degree $2q$ summand of the above complex
\[
\pi_{2q}(NMU)^*:=[\L\to  \L\otimes\Z[b_1,b_2,\ldots]_+\to\ldots\to \L\otimes(\Z[b_1,b_2,\ldots]_+)^{\otimes n}\to\ldots]_{2q};
\]
note that $[\L\otimes(\Z[b_1,b_2,\ldots]_+)^{\otimes m}]_{2q}=0$ for $m>q$, so $\pi_{2q}(NMU)^*$ is supported in cohomological degrees $[0,q]$. 

\begin{conj}[Voevodsky \hbox{\cite[conjecture 9]{VoevOpen}}] \label{conj:VoevSlice}There is a natural isomorphism in $\SH(k)$
\[
s_q(\mS_k)\cong \Sigma^q_T \EM(\Z^{tr}\otimes \pi_{2q}(NMU)^*).
\]
\end{conj}
Here $\Z^{tr}=\Z^{tr}(\Spec k)\in \DM(k)$. 

The conjecture immediately implies
\begin{cor}\label{cor:SliceAN} 1. $\pi^\mu_q(\mS_k)\cong \Z^\tr\otimes\pi_{2q}(NMU)^*$.\\
2.  The cohomology sheaves $\sH^p(\pi^\mu_q(\mS_k))$ of the effective motive $\pi^\mu_q(\mS_k)$ are zero for $p<0$, $p>q$.\\
3. For each  $q>0$ and each $p$, $0\le p\le q$ there is a finite abelian group $A_{p,q}$ with $\sH^p(\pi^\mu_q(\mS_k))\cong A_{p,q}\otimes\Z^\tr$. \\
4. $\pi^\mu_0(\mS_k)=\Z^\tr$.
\end{cor}

The group $A_{p,q}$ is just the $E_2^{p,-2q}$ term in the Adams-Novikov spectral sequence
\[
A_{p,q}=\Ext^{p, 2q}_{MU_*(MU)}(MU_*, MU_*)=E_2^{p, -2q}(AN).
\]
This follows directly from the identification of $E_2^{p,q}(AN)$ with $H^p(\pi_{-q}(NMU)^*)$ (see e.g. \cite[III, \S 15]{Adams}, here we use the   indexing convention for which  $E_2^{p,q}(AN)$ contributes to $\pi_{-p-q}(\mS)$).   

Corollary~\ref{cor:SliceAN} hints at a possible connection between the Atiyah-Hirzebruch spectral sequence associated to the slice tower for $\mS_k$:
\[
E_2^{p,q}(AH)=H^{p-q}(\Spec k, \pi^\mu_{-q}\mS_k(-q))\Longrightarrow \Pi_{-p-q,0}(\mS_k)(k),
\]
and the Adams-Novikov spectral sequence. In fact, we have 
\begin{thm}\label{thm:SliceAN} For $k$ algebraically closed of characteristic zero we have
\[
E_2^{p,q}(AH)= E_2^{p-q, 2q}(AN)\otimes\hat{\Z}(-q)
\]
where $\hat{\Z}(q)=\lim_N\mu_N^{\otimes q}$.
\end{thm} 
This is  theorem~\ref{IntroThm:Main3}, announced in the introduction; we reiterate that we do not know if $d_3(AN)=d_2(AH)$, even though these two differentials have isomorphic source and target.

\begin{proof} 
Since $k$ is algebraically closed and of characteristic zero, the Suslin-Voevodsky theorem  \cite[theorem 8.3]{SuslinVoev} implies (for $q\ge0$)
\[
H^n(\Spec k, \Z/N(q))=\begin{cases}0&\text{ for }n\neq0\\ \mu_N^{\otimes q}&\text{ for }n=0.\end{cases}
\]
Thus the spectral sequence 
\[
E_2^{a,b}=H^a(\Spec k, \sH^b(\pi^\mu_{q}\mS_k)(q))\Longrightarrow H^{a+b}(\Spec k, \pi^\mu_{q}\mS_k(q))
\]
degenerates at $E_2$, $E_2^{a,b}=0$ for $a\neq0$,  and we have
\[
E_2^{p,-q}(AH)=H^{p+q}(\Spec k, \pi^\mu_{q}\mS_k(q))=A_{p+q,q}\otimes\hat{\Z}(q)=E_2^{p+q,-2q}(AN)\otimes\hat{\Z}(q).
\]
\end{proof}

\begin{proof}[Proof of conjecture~\ref{conj:VoevSlice}] 
We adapt the construction of the Adams-Novikov spectral sequence given in \cite[\hbox{\it loc. cit.}]{Adams}. This involves the use of $n$-cubes in $\Spt_T(k)$; in order to deal with these, we need a functorial version of the slices $s_q$ in the homotopy category of $M\Z$-modules, which we now proceed to construct.

Let $\sS$ be the category associated to a finite partially ordered set and $\sM$ a pointed complete and cocomplete category, giving us the functor category $\sM^\sS$. For $s\in\sS$, let $i_s^*:\sM^\sS\to \sM$ be the evaluation at $s$. Let ${\bf F}^s:\sM\to \sM^\sS$ be the free diagram functor at $s$ \cite[definition 11.5.25]{Hirschhorn}; as $\sS$ is a partially ordered set, ${\bf F}^s(A)$ is the constant functor with value $A$ on the subcategory $\sS_{\ge s}$ of objects $t\ge s$ in $\sS$, extended by $\pt$ to the rest of $\sS$, and similarly for morphisms. We have as well the ``dual" ${\bf F}_s:\sM\to \sM^\sS$ sending $A$ to the constant functor with value $A$ on the subcategory $\sS_{\le s}$ of objects $t\le s$ in $\sS$, extended by $\pt$ to the rest of $\sS$. ${\bf F}^s$ is left adjoint to $i_s^*$ and ${\bf F}_s$ is right adjoint. If $\sM$ is a pointed symmetric monoidal category and $\sS$ has an initial object $0$,  we make $\sM^\sS$ a pointed symmetric monoidal category with $(F\wedge G)(s)=F(s)\wedge G(s)$ and unit $\1^\sS:={\bf F}^0(\1)$; in this case   $i_s^*$ is monoidal and ${\bf F}^s$ and ${\bf F}_s$ are monoidal except for preserving the unit.

Suppose $\sM$ is a pointed model category. We give $\sM^\sS$ the projective model structure (proposition~\ref{prop:FunctorModelCat}). The next result lists a number of properties of $\sM^\sS$.

\begin{lem}\label{lem:FunctCatProps} 1. The functors $i_s^*$ and ${\bf F}^s$ preserve cofibrations, fibrations and weak equivalences, ${\bf F}_s$ preserves fibrations and weak equivalences, and the adjoint pairs $({\bf F}^s, i_s^*)$ and $(i_s^*, {\bf F}_s)$ are Quillen pairs.\\
2. If $\sM$ is a simplicial, resp.\,left proper, resp.\,right proper, resp.\,cofibrantly generated, resp.\,cellular, resp.\,combinatorial, pointed model category, the same holds for $\sM^\sS$. If $I$ generates the cofibrations in $\sM$ and $J$ generates the trivial cofibrations in $\sM$, then the collection ${\bf F}^s(I)$, $s\in \sS$, generates the cofibrations in $\sM^\sS$ and  ${\bf F}^s(J)$, $s\in \sS$, generates the trivial cofibrations in $\sM^\sS$.\\
3. Suppose that $\sS$ has an initial object and $\sM$ is a pointed (simplicial) monoidal cofibrantly generated model category. Then $\sM^\sS$ is a pointed (simplicial) monoidal model category.
\end{lem}

\begin{proof} For (1), it is obvious that all three functors preserve fibrations and weak equivalences; as $i_s^*$ and ${\bf F}^s$ are left adjoint to ${\bf F}_s$ and $i_s^*$, respectively, this shows that $i_s^*$ and ${\bf F}^s$ preserve cofibrations and $({\bf F}^s, i_s^*)$ and $(i_s^*, {\bf F}_s)$ are Quillen pairs.

For (2), the assertions about cofibrantly generated, resp. cellular, resp. combinatorial $\sM$ are proven in \cite[theorem 11.6.1, proposition 12.1.5]{Hirschhorn}; the statement on combinatorial model categories is \cite[theorem 2.14]{Barwick}. The proof of \cite[theorem 11.7.3]{Hirschhorn} shows that simplicial structure on $\sM$ makes $\sM^\sS$ into a simplicial model category. As fibrations,  weak equivalences and pull-backs in $\sM^\sS$ are defined pointwise, the statement about right properness is clear. Similarly, as by (1), every cofibration in $\sM^\sS$ is a pointwise cofibration, the fact that weak equivalences and push-outs are defined pointwise shows that $\sM^\sS$ inherits left properness from $\sM$.

For (3), define the internal Hom for $X, Y\in \sM^\sS$ by the equalizer sequence
\[
\sHom(X,Y)(t)\to \prod_{s\ge t}\sHom(X(s), Y(s))\xymatrix{\ar@<3pt>[r]\ar@<-3pt>[r]&} \prod_{s_2\ge s_1\ge t}\sHom(X(s_1), Y(s_2))
\]
We verify the axioms \cite[definition 4.2.6]{Hovey}:  4.2.6(1) follows from the description of the generating cofibrations and generating trivial cofibrations in $\sM^\sS$ given in (2) and \cite[corollary 4.2.5]{Hovey} and  4.2.6(2) follows from the fact that evaluation at all $s\in \sS$ preserves cofibrations and detects weak equivalences.
\end{proof}

\begin{lem} \label{lem:Slice} Suppose that $\sS$ has an initial object. \\
1. For each $q\in \Z$, there is a functor $\tilde{s}_q:\Ho\,\Spt_T^\Sigma(k)\to \Ho\,\Mod\text{-}M\Z$ and a natural isomorphism $RF\circ \tilde{s}_q\cong s_q$.\\
2.  For each $q\in\Z$, there is a functor $\tilde{s}_q^\sS:\Ho\,(\Spt_T^\Sigma(k)^\sS)\to :\Ho\,(\Mod\text{-}M\Z^\sS)$ and natural isomorphisms for $c\in\sS$, $\tilde{s}_q\circ i^*_c\cong i_c^*\circ \tilde{s}_q^\sS$.
\end{lem}

\begin{proof} 
(1) follows from \cite[theorem 5.2]{GRSO} applied  to the $E_\infty$ object $\mS^\Sigma_k$ in the combinatorial monoidal stable model category\footnote{$\Spt^\Sigma_T(k)$  is  combinatorial: use \cite[thm. 2.6.15]{Pelaez} and the definition of $\Spt^\Sigma_T(k)$ as the category of $T^{\wedge *}$-modules in symmetric sequences in the presheaf category $\Spc_\bullet(k):=(\Spc_\bullet)^{\Sm/k^\op}$.} $\Spt^\Sigma_T(k)$, and the sequence of full subcategories $\sC_q:=\tau^{-\infty,q}\SH(k)$ of $\SH(k)$. In order to apply the theorem, one uses Voevodsky's isomorphism $s_0\mS_k\cong M\Z$, and notes that 
$\sC_q$ satisfies the condition (A3) of  \cite{GRSO}  by theorem~\ref{thm:RBousLoc} and theorem~\ref{thm:truncation}; the remaining conditions are easy to verify (see \cite[\S 3.2]{GRSO}). In order to prove (2), we briefly recall the main points of the construction in  \cite{GRSO}:
\\
A) Let $\sD_1\subset \sD_0\subset \sD:= (\Ho\,\Spt_T^\Sigma(k))^\Z$ be the subcategories $\sD_0:=\prod_q\sC_q$ and $\sD_1:=\prod_q\sC_{q+1}$. The functor $c_0:\sD\to \sD_0$ is constructed as the right derived functor associated to a right Bousfield localization of  $(\Spt_T^\Sigma(k))^\Z$ and the functor $l_1:\sD_0\to \sD_0$ is the restriction to $\sD_0$ of the localization of $\sD$ with respect to $\sD_1$, which is realized as the left derived functor of a left Bousfield localization (again, of $(\Spt_T^\Sigma(k))^\Z$). Composing $l_1\circ c_0$ with the diagonal functor  $\Ho\,\Spt_T^\Sigma(k)\to (\Ho\,\Spt_T^\Sigma(k))^\Z$ gives the slice functor
\[
s_*=\prod_{q\in\Z}s_q:\Ho\,\Spt_T^\Sigma(k)\to (\Ho\,\Spt_T^\Sigma(k))^\Z.
\]
B) Let $E_\infty$ be the simplicial $E_\infty$ operad. It is shown that both the cofibrant replacement with respect to the right Bousfield localization in (A) and the fibrant replacement with respect to the left Bousfield localization in (A) induce a trivial fibration on the relevant simplicial mapping spaces (for the cofibrant replacement $Q$, one considers the map $\sHom(QE^{\otimes n}, QE)\to \sHom(QE^{\otimes n}, E)$ for $E$ an $E_\infty$-algebra in $\Spt_T^\Sigma(k)$ and a similar collection of maps for the fibrant replacement): this gives a unique up-to-homotopy lifting of the $E_\infty$-structure on $\1$ to an $E_\infty$-structure on  $s_*(\1)$.\\
C) Replacing the $E_\infty$ operad with the colored operad controlling modules over an $E_\infty$-algebra, it is shown that the cofibrant and fibrant replacement functors induce trivial fibrations on the relevant mapping spaces; this gives a canonical (up to homotopy) $s_*(\1)$-module structure to the functor $s_*$, which gives the desired lifting of $s_*$ to 
\[
\tilde{s}_*:\Ho\,\Spt_T^\Sigma(k)\to  \Ho\,\gr\Mod\text{-} s_*(\1);
\]
that is,  $RF^\Z\circ \tilde{s}_*\cong s_*$, where $F^\Z:\gr\Mod\text{-} s^\sS_*(\1)\to \Spt_T(k)^\Z$ is the evident forgetful functor. Restricting 
a graded $s_*(\1)$-module to a graded $s_0(\1)$-module and using Voevodsky's isomorphism $s_0(\1)\cong M\Z$ gives the functors $\tilde{s}_q$.

For (2), we use   \cite[\hbox{\it loc.\,cit.}]{GRSO} applied  to the combinatorial monoidal stable model category $\Spt_T^\Sigma(k)^\sS$, the $E_\infty$-object 
$\1^\sS={\bf F}^0(\mS^\Sigma_k)$ and the sequence of full subcategories $\sC^\sS_q$ of $\Ho\,(\Spt_T^\Sigma(k)^\sS)$,  with $\sC^\sS_q$ defined as the localizing subcategory of $\Ho\,(\Spt_T^\Sigma(k)^\sS)$ generated by the objects ${\bf F}^s(K_{-\infty, q})$, $s\in \sS$. Applying \cite[\hbox{\it loc.\,cit.}]{GRSO} gives the $E_\infty$-object $s^\sS_*(\1^\sS)$ in $\Ho\,(\Spt_T^\Sigma(k)^\sS)^{\N}$ and the functor 
\[
\tilde{s}^\sS_*:\Ho\,(\Spt_T^\Sigma(k)^\sS)\to\Ho\,\gr\Mod\text{-} s^\sS_*(\1^\sS).
\]
We let $\tilde{s}_q^\sS: \Ho\,(\Spt_T^\Sigma(k)^\sS)\to\Ho\,\Mod\text{-} s^\sS_0(\1^\sS)$ be the $q$th component of $\tilde{s}^\sS_*$.

The functor $\tilde{s}^\sS_*$ and the $E_\infty$-object $s^\sS_*(\1^\sS)$ are constructed using the same three steps (now named  (A)$^\sS$-(C)$^\sS$) after making the replacements described above. We write $c_0^\sS$, $l_1^\sS$, etc., for the corresponding constructions in this case.

Take $c\in \sS$.  Since $i_c^*$ is both a left and a right Quillen functor,    $i_c^*(\sC_q^\sS)=\sC_q$ for all $q$, $i_c^*\circ c_0^\sS\cong c_0\circ i_c^*$  and $i_c^*\circ l_1^\sS\cong l_1\circ i_c^*$,  giving a canonical isomorphism  $s_*\circ i_c^*\cong i_c^*\circ s^\sS_*$.
 
We have $\1=i_c^*(\1^\sS)$. In addition, the mapping spaces considered in (B), (B)$^\sS$, (C) and (C)$^\sS$ (as functors in the $E_\infty$-object $E$, module $M$ and the fibrant and cofibrant replacements) only depend on the operads chosen, and hence $i^*_c$ gives morphism from the maps shown to be a trivial fibration in (B)$^\sS$ and (C)$^\sS$ to the analogous ones in (B) and (C). By the up-to-homotopy uniqueness of the lifting of $E_\infty$-structures in  step (B), it follows  that $i_c^*(s_*^\sS(\1^\sS))$ is homotopy equivalent to $s_*(\1)$ as a $E_\infty$-object in $\Spt_T(k)^{\N}$, which thus gives an equivalence of $\Ho\,\gr\Mod\text{-} i_c^*s^\sS_*(\1^\sS)$ with $\Ho\,\gr\Mod\text{-} s_*(\1)$.  Via this equivalence, the up-to-homotopy uniqueness in (C) gives a canonical isomorphism of $i_c^*\tilde{s}^\sS_*(\sE)$ with 
$\tilde{s}_*(i_c^*(\sE))$ in $\Ho\,\gr\Mod\text{-} s_*(\1)$ for all $\sE\in \Spt_T^\Sigma(k)^\sS$; taking the restriction to graded $M\Z$-modules completes the construction of the natural isomorphism  $\tilde{s}_q\circ i^*_c\cong i_c^*\circ \tilde{s}_q^\sS$.
\end{proof}

We now return to the Adams-Novikov spectral sequence. Consider the distinguished triangle
\begin{equation}\label{eqn:ReducedMGL}
 \overline{\MGL}[-1]\to \mS_k\to \MGL\to \overline{\MGL} 
\end{equation}
Using the cell structure of $\MGL$, it is easy to see that the unit map $ \mS_k\to \MGL$ induces an isomorphism $s_0\mS_k\to s_0\MGL$ (see e.g. \cite[corollary 3.3]{Spitzweck}). Since $\MGL$ and $\mS_k$ are both in $\SH^\eff(k)$, it follows that $\overline{\MGL}$ also in $\SH^\eff(k)$ and that $s_0 \overline{\MGL} =0$. Thus $\overline{\MGL}$ is in $\Sigma_T\SH^\eff(k)$ and hence $\overline{\MGL}^{\wedge N}$ is in  $\Sigma^N_T\SH^\eff(k)$ for each $N\ge1$.

Let $\sq^n$ be the category associated to the partially ordered set  of subsets of $\{1,\ldots, n\}$, ordered by inclusion, $\sq^n_0$ the subcategory of non-empty subsets. By an $n$-cube in a category $\sC$, we mean a functor from $\sq^n$ to $\sC$. For an $n$-cube $I\mapsto \sE_I$ in $\Spt^\Sigma_T(k)$, we have the map of $n-1$-cubes $I\mapsto [\sE_{I}\to \sE_{I \amalg\{n\}}]$, $I\subset \{1,\ldots, n-1\}$.  We form the $T$-spectrum $\Tot_n\sE_*$ inductively in $n$ as the homotopy fiber of  $\Tot_{n-1}\sE_{*}\to \Tot_{n-1}\sE_{I\amalg\{n\}}$; $\Tot_0(\sE):=\sE$. We make a similar definition for $n$-cubes in $\Spt_T(k)$, $\Spt$ or $C(\Ab)$. 

Form the product  $[\mS_k\to  \MGL]^{\wedge n}$
as an $n$-cube in $\Spt^\Sigma_T(k)$,  giving us the object $\Tot[\mS_k\to  \MGL]^{\wedge n}$ in $\Spt^\Sigma_T(k)$. The distinguished triangle \eqref{eqn:ReducedMGL} defines an isomorphism of 
$\Tot[\mS_k\to  \MGL]^{\wedge n}$ with $\overline{\MGL}^{\wedge n}[-n]$ in $\SH(k)$.  In particular, we have 
\[
\tilde{s}_q\Tot[\mS_k\to  \MGL]^{\wedge n}\cong 0
\]
for $0\le q\le n-1$. 

Let $[\mS_k\to  \MGL]_0^{\wedge n}$ be the  $n$-cube formed from $[\mS_k\to  \MGL]^{\wedge n}$ by replacing the $\mS_k$ located at the vertex $\0$ with the 0-object. We thus have the homotopy cofiber sequence
\[
\Tot[\mS_k\to  \MGL]_0^{\wedge n}\to \Tot[\mS_k\to  \MGL]^{\wedge n}\to \Tot [\mS_k\to  0]^{\wedge n}
\]
As $\Tot [\mS_k\to  0]^{\wedge n}$ is isomorphic in $\SH(k)$ to $(\mS_k)^{\wedge n}=\mS_k$, this gives us the distinguished triangle in $\SH(k)$
\[
\mS_k\to \Tot[\mS_k\to  \MGL]_0^{\wedge n}[1]\to \overline{\MGL}^{\wedge n}[1]\to \mS_k[1]
\]
In particular, we have the isomorphism in $\Ho\,\Mod\text{-}M\Z$
\begin{equation}\label{eqn:MGLSlice1}
\tilde{s}_q\mS_k\cong \tilde{s}_q\Tot[\mS_k\to  \MGL]_0^{\wedge n}[1]
\end{equation}
for $0\le q<n$.

As $\tilde{s}_q$ is exact, we have
\begin{equation}\label{eqn:MGLSlice2}
\tilde{s}_q\Tot[\mS_k\to  \MGL]_0^{\wedge n}[1]\cong
\Tot \tilde{s}^\sS_q[\mS_k\to  \MGL]_0^{\wedge n}[1]
\end{equation}
Here   we use the functorial model $\tilde{s}_q^\sS$ for $\tilde{s}_q$ furnished by lemma~\ref{lem:Slice}. Furthermore, the value of $[\mS_k\to  \MGL]_0^{\wedge n}$ at $I\neq \0$ is $\MGL^{\wedge |I|}$, so   $\tilde{s}^\sS_q[\mS_k\to  \MGL]_0^{\wedge n}(I)=\tilde{s}_q(\MGL^{\wedge |I|})$. 

For a complex of abelian groups $C$, write $M\Z\otimes C$ for $M\Z\wedge_{H\Z}EM(C)$; via the R\"ondigs-{\O}stv{\ae}r equivalence $DM(k)\cong \Ho\,M\Z\text{-}\Mod$, $M\Z\wedge_{H\Z}EM(C)$ corresponds to $\Z^{tr}\otimes_\Z C$. We now apply the theorem of Hopkins-Morel \cite{Hoyois, HopkinsMorel}:
\begin{thm}[Hopkins-Morel,  \hbox{\cite[theorem 7.5]{Hoyois}}] There is an isomorphism 
\[
\tilde{s}_q\MGL\cong \Sigma^q_TM\Z\otimes MU_{2q}
\]
in $\Ho\,\Mod\text{-}M\Z$.
\end{thm}

\begin{proof} In fact, the  theorem of Hopkins-Morel says that $s_q\MGL$ is isomorphic  to  $\Sigma^q_T\EM (\Z^{tr}\otimes MU_{2q})$ in $\SH(k)$. To achieve the isomorphism in $\Ho\,\Mod\text{-}M\Z$, we note the following:

\begin{lem}\label{lem:lifting} For $\sE\in\SH(k)$, there is a canonical isomorphism $\tilde{s}_q(s_q\sE)\cong \tilde{s}_q(\sE)$ in $\Ho\,\Mod\text{-}M\Z$.
\end{lem}

\begin{proof} Apply $\tilde{s}_q$ to the diagram $s_q\sE\leftarrow f_q\sE\to \sE$, giving the diagram 
\[
\tilde{s}_q(s_q\sE)\xleftarrow{\alpha}
 \tilde{s}_q(f_q\sE)\xrightarrow{\beta} \tilde{s}_q(\sE)
 \]
  in $\Ho\,\Mod\text{-}M\Z$. Applying $RF$ gives $s_q(s_q\sE)\leftarrow s_q(f_q\sE)\to s_q(\sE)$, which is easily seen to be a diagram of isomorphisms (in $\SH(k)$). But by definition of the weak equivalences in $\Mod\text{-}M\Z$, the map $RF$ detects isomorphisms, hence $\alpha$ and $\beta$ are isomorphisms in $\Ho\,\Mod\text{-}M\Z$.
 \end{proof}
 
 To complete the proof of the refined version of the Hopkins-Morel theorem, $s_0(\mS_k)\cong M\Z\cong \EM(\Z^{tr})$, hence $M\Z=\tilde{s}_0(\mS_k)\cong \tilde{s}_0(\EM(\Z^{tr}))$. Thus $\tilde{s}_q(\Sigma^q_T\EM(\Z^{tr}))\cong \Sigma^q_TM\Z$; as $MU_{2q}$ is a free finitely generated abelian group, $\Sigma^q_T\EM (\Z^{tr}\otimes MU_{2q})$ is just a finite direct sum of copies of $\Sigma^q_T\EM(\Z^{tr})$, giving the string of isomorphisms in $\Ho\,\Mod\text{-}M\Z$:
 \[
 \tilde{s}_q(\MGL)\cong \tilde{s}_q(s_q\MGL)\cong \tilde{s}_q(\Sigma^q_T\EM (\Z^{tr}\otimes MU_{2q}))\cong \Sigma^q_TM\Z\otimes MU_{2q}.
 \]
 \end{proof}

Using \cite[proposition 6.4]{Spitzweck} or \cite[theorem 7.5]{Hoyois} and applying lemma~\ref{lem:lifting},  the Hopkins-Morel theorem generalizes to give the isomorphism in $\Ho\,\Mod\text{-}M\Z$
\begin{equation}\label{eqn:HMIso}
\tilde{s}_q\MGL^{\wedge j}\cong \Sigma^q_TM\Z\otimes \pi_{2q}(MU^{\wedge j}).
\end{equation}

\begin{lem} Let $\<\Sigma^q_T M\Z\>$ denote the thick subcategory of  $\Ho\,\Mod\text{-}M\Z$ generated by 
$\Sigma^q_T M\Z$. Sending a complex $C\in C^b(\Ab)$ to $\Sigma^q_T M\Z\otimes C$ defines an equivalence of triangulated categories
$\Sigma^q_T M\Z\otimes(-):D^b(\Ab)\to \<\Sigma^q_T M\Z\>$
\end{lem}

\begin{proof} Under the equivalence $\Ho\,\Mod\text{-}M\Z\cong \DM(k)$, $\Sigma^q_T M\Z$ gets sent to $\Z(q)[2q]$, and thus
\[
\Hom_{\Ho\,\Mod\text{-}M\Z}(\Sigma^q_T M\Z,\Sigma^q_T M\Z[n])= H^n(k,\Z(0))\cdot\id=\begin{cases}0&\text{ for }n\neq0\\\Z\cdot\id&\text{ for }n=0,\end{cases}
\]
from which the lemma follows.
\end{proof}

Using \eqref{eqn:HMIso} and this lemma, we may consider the $n$-cube $\tilde{s}^\sS_q[\mS_k\to  \MGL]_0^{\wedge n}$ as an $n$-cube in $D^b(\Ab)$; noting that each term in this  $n$-cube is actually in the heart of  $D^b(\Ab)$ for the standard $t$-structure, we may consider the  $n$-cube $\tilde{s}^\sS_q[\mS_k\to  \MGL]_0^{\wedge n}$ as an $n$-cube in $\Ab$, that is, we have an isomorphism  in $\Ho\,(\Mod\text{-}M\Z)^\sS$
\[
\tilde{s}^\sS_q[\mS_k\to  \MGL]_0^{\wedge n}\cong \Sigma^q_T M\Z\otimes (\pi_{2q}[\mS\to  MU]_0^{\wedge n})
\]
where $\pi_{2q}[\mS\to  MU]_0^{\wedge n}$ is the  $n$-cube in $\Ab$ formed by applying $\pi_{2q}$ termwise to the  $n$-cube $[\mS\to  MU]_0^{\wedge n}$ in $\Spt$.   This in turn gives the isomorphism in $\Ho\,\Mod\text{-}M\Z$
\begin{equation}\label{eqn:MGLSlice3}
\Tot\,\tilde{s}^\sS_q[\mS_k\to  \MGL]_0^{\wedge n}[1]\cong \Sigma^q_TM\Z\otimes(\Tot(\pi_{2q}[\mS\to  MU]_0^{\wedge n})[1]).
\end{equation}

We complete the proof of Voevodsky's conjecture by constructing a quasi-iso\-morphism $\beta: \pi_{2q}(NMU)^*\to \Tot(\pi_{2q}[\mS\to  MU]_0^{\wedge n})[1]$. 
This follows from a general fact about cosimplicial abelian groups. Namely, define the functor $p_n:\sq^n_0\to \Delta$ by identifying a non-empty subset $I$ of $\{1,\ldots, n\}$ with the ordered set $[|I|-1]$ via the unique order-preserving bijection $I\to [|I|-1]$, where we give $I$ the order induced by the opposite of the standard order on $\{1,\ldots, n\}$.  Given a cosimplicial abelian group $A^*:\Delta\to \Ab$ and an integer $n\ge1$, we may then form the $n$-cube of abelian groups $\sq^n(A^*)$ by composing $A^*$ with $p_n$ and filling in by setting  $\sq^n(A^*)(\0)=0$.

The following result is standard: Let $i_j:NA^j\to A^j$ be the inclusion.  For $j<n$,  let 
\[
\beta^j:NA^j\to \Tot_n\sq^n(A^*)[1]^j=\oplus_{I\subset \{1,\ldots, n\}, |I|=j} A^j
\]
be the map $(i_j,\ldots, i_j)$,  and let $\beta^j=0$ for $j\ge n$. 

\begin{lem}\label{lem:DK} The maps $\beta^j$ define a map of complexes
\[
\beta: NA^* \to \Tot_n\sq^n(A^*)[1]
\]
which is an isomorphism on $H^p$ for $0\le p<n-1$ and an injection for $p=n-1$. 
\end{lem}

Fix an integer $q\ge0$ and take $n$ to be any integer $n\ge q+2$. By lemma~\ref{lem:DK}, we have  maps of complexes
\begin{equation}\label{eqn:Complexes}
\beta:\pi_{2q}(NMU)^*\to \Tot \pi_{2q}[\mS\to  MU]_0^{\wedge n}[1]
\end{equation}
which is a   cohomology isomorphism  in degrees $\le n-2$. Also, $\pi_{2q}(NMU)^m=0$ for $m>q$. 

Let
$e_i:\pi_{2q}MU^{\wedge n-1}\to \pi_{2q}MU^{\wedge n}$ be the map induced by inserting the unit in the $i$th factor. By reason of degree, the map $\sum_{i=0}^n e_i:\oplus \pi_{2q}MU^{\wedge n-1}\to  \pi_{2q}MU^{\wedge n}$
is surjective. Thus, $H^{n-1}(\Tot \pi_{2q}[\mS\to  MU]_0^{\wedge n}[1])=0$, and   $\beta$ gives us our desired quasi-isomorphism.  

Combining this with \eqref{eqn:MGLSlice1}, \eqref{eqn:MGLSlice2} and \eqref{eqn:MGLSlice3} completes the proof of conjecture~\ref{conj:VoevSlice}.

 \end{proof}

\begin{appendix}
\section{Symmetric products}\label{sec;sym}
In this appendix, we discuss symmetric products in $\Spc(k)$. This follows Voevodsky's constructions in \cite{VoevMotivic}, but we use a less sophisticated approach, in that we do not consider any model category structures or use derived functors. For simplicity, we work over a field $k$ of characteristic zero.
 
Let $\Sch/k$ be the category of quasi-projective $k$-schemes and $\sC\subset \Sch/k$ be the  full subcategory of {\em connected} semi-normal quasi-projective $k$-schemes. We let $\Spc^\sC(k)$ denote the category of  presheaves of spaces on $\sC$, and $\Spc^\sC_\bullet(k)$ the category of  presheaves of pointed spaces on $\sC$. For $T\in \sC$ with a finite group $G$ acting on $T$, the quotient scheme $T/G$ exists and is in $\sC$. For $X\in \sC$, $n\ge0$ an integer, we have the $k$-scheme $\Sym^n X:=X^n/\Sigma_n$, where the symmetric group $\Sigma_n$ acts on $X^n$ by  permuting the factors.  For a pointed scheme $(X,x)$, we make $\Sym^n X$ a pointed scheme  with $\Sym^n x$ as base-point, denoted $\Sym^n_\bullet X$.

 To extend this to   $\sX\in \Spc^\sC(k)$, let  $(\sC, G)$ be the category of finite type  $k$-schemes $X$  with $G$-action, $G$ a finite group, such that each connected component of $X$ is in $\sC$ and $X/G$ is in $\sC$. This gives us  the corresponding presheaf category $\Spc^{(\sC,G)}(k)$.  We write $h_\sC:\sC^\op\to \Spc^\sC(k)$, $h_{(\sC,G)}:(\sC,G)^\op\to  \Spc^{(\sC,G)}(k)$ for the Yoneda embeddings.

The functor $triv:\sC\to (\sC,G)$ giving   $X\in\sC$ the trivial $G$-action yields the functor $triv_*:\Spc^{(\sC,G)}(k)\to \Spc^\sC(k)$, $triv_*(\sY):=\sY\circ triv$,  with left adjoint $triv^*:\Spc^\sC(k)\to \Spc^{(\sC,G)}(k)$. In fact, $triv^*=\pi_{G*}$, with $\pi_G:(\sC,G)\to \sC$ the functor $\pi_G(Y)=Y/G$. Therefore, $triv^*$ admits in turn a left  adjoint $triv_\#:\Spc^{(\sC,G)}(k)\to \Spc^\sC(k)$, this being the left Kan extension of $h_\sC\circ \pi_G^\op$.  We write $\sX/G$ for $triv_\#(\sX)$.  
 
Let $\sC^\natural$ be the full subcategory of $\Sch/k$ of semi-normal schemes, and $(\sC^\natural, G)$ the category of semi-normal schemes with $G$-action. We extend $\sX$ in $\Spc^\sC(k)$ to a presheaf on $\sC^\natural$ by defining $\sX(\amalg_iX_i):=\prod_i\sX(X_i)$. Let  $G\text{-}\Spc^\sC(k)$ be the category of presheaves of spaces with $G$-action on $\sC$. We let $G\text{-}h_\sC:(\sC,G)^\op\to G\text{-}\Spc^\sC(k)$ be the functor sending $Y\in (\sC,G)$ to  the representable presheaf  $h_\sC(Y)$, with $G$-action induced by the action on $Y$.

We extend $h_\sC$ to $\Sch/k^\op$ by letting $h_\sC(X)$ be the restriction to $\sC$ of the presheaf on $\Sch/k$ represented by $X$. As each $Y\in\sC$ is connected, $h_\sC$ sends disjoint union in $\Sch/k$ to coproducts in $\Spc^\sC(k)$.  We similarly extend $G\text{-}h_\sC$ and $h_{\sC,G}$ to  (the opposite of) the category of quasi-projective $G$-schemes,  $(\Sch/k, G)$.

For $\sX\in G\text{-}\Spc^\sC(k)$, define $\sX^G\in \Spc^{(\sC,G)}(k)$ as the presheaf 
\[
\sX^G(Y):=\Hom_{G\text{-}\Spc^\sC(k)}(G\text{-}h_\sC(Y), \sX), 
\]
giving the functor $(-)^G:G\text{-}\Spc^\sC(k)\to \Spc^{(\sC,G)}(k)$.
Note that $\sX^G(Y)=[\sX(Y)]^G$, where $G$ acts on $\sX(Y)$ by $g\cdot s=g_\sX\cdot (s\circ g_Y^{-1})$, with $g_\sX$ the action of $g$ on $\sX$ and $g_Y$ the action on $Y$.

 For $\sX\in G\text{-}\Spc^\sC(k)$, $\sX^G/G$ may be described as a colimit:
\begin{equation}\label{eqn:Colim}
(\sX^G/G)(X)=  \colim_{(Y,f:X\to Y/G)\in X/\pi_G} \sX^G(Y).
\end{equation}

Given finite groups $G_1, G_2$, we have the evident product functor
\[
\times:G_1\text{-}\Spc^\sC(k)\times G_2\text{-}\Spc^\sC(k)\to G_1\times G_2\text{-}\Spc^\sC(k);
\]
\eqref{eqn:Colim} gives the natural morphism
\begin{equation}\label{eqn:Product}
\sX_1^{G_1}/G_1\times \sX_2^{G_2}/G_2\to (\sX_1\times \sX_2)^{G_1\times G_2}/G_1\times G_2.
\end{equation}

For $\rho:H\to G$ a homomorphism of finite groups, we have the restriction-of-action functor $\rho^*:G\text{-}\Spc^\sC(k)\to H\text{-}\Spc^\sC(k)$,  and for $Y\in(\sC,H)$, we have  the induced $G$-scheme $\ind_H^GY:=G\times Y/H$, where  $H$ acts by $h\cdot(g,y):=(g\rho(h)^{-1}, h_Y(y))$. Using the isomorphisms 
\[
\sX^G(\ind_H^GY)\cong (\rho^*\sX)^H(Y),\ \ind_H^GY/G\cong Y/H, 
\]
\eqref{eqn:Colim} gives for $\sX\in G\text{-}\Spc^\sC(k)$ the natural map
\begin{equation}\label{eqn:Restriction}
(\rho^*\sX)^H/H\to \sX^G/G.
\end{equation}

We have   the functor $(-)^n:\Spc^\sC(k)\to \Sigma_n\text{-}\Spc^\sC(k)$ sending $\sX$ to $\sX^n$ with $\Sigma_n$-action permuting the factors. Define
\[
\Sym^n\sX:= (\sX^n)^{\Sigma_n}/\Sigma_n.
\]
This gives us a functor $\Sym^n:\Spc^\sC(k)\to \Spc^\sC(k)$; for $(\sX,x)$ a pointed space, define $\Sym^n_\bullet(\sX)$ to be $\Sym^n(\sX)$ (i.e., forget the base-point), pointed by $\Sym^nx$. We may restrict (via the inclusion $\Sm/k\hookrightarrow\sC^\natural$) to a presheaf on $\Sm/k$ to give $\Sym^n\sX\in \Spc(k)$, resp., $\Sym^n_\bullet\sX\in \Spc_\bullet(k)$.

Using  \eqref{eqn:Product} and \eqref{eqn:Restriction}, one constructs canonical sum and product maps 
\[
\Sym^n_\bullet\sX\times \Sym^m_\bullet\sX\to \Sym^{n+m}_\bullet\sX;\quad \Sym^n_\bullet\sX\wedge \Sym^m_\bullet\sY\to \Sym^{nm}_\bullet\sX\wedge\sY.
\]
Adding the base-point thus gives the sequence of ``stabilization" maps
\[
\sX=\Sym^1_\bullet(\sX)\xrightarrow{st_2}\Sym^2_\bullet(\sX)\xrightarrow{st_3}\ldots \xrightarrow{st_n}\Sym^n_\bullet(\sX)\xrightarrow{st_{n+1}}\ldots
\]
and the infinite symmetric product $\Sym^\infty_\bullet\sX:=\colim_n\Sym^n_\bullet\sX$.

\begin{lem}\label{lem:FiniteTypeSym}  
1. Let $W$ be a quasi-projective $k$-scheme with $G$-action. Then $G\text{-}h_\sC(W)^G\in \Spc^{\sC,G}(k)$ is represented by  the $G$-scheme $W$ and $(G\text{-}h_\sC(W)^G)/G$ is represented by $W/G$.\\
2. Let $Z$ be a quasi-projective $k$-scheme. Then the scheme $\Sym^nZ$ represents $\Sym^nh_\sC(Z)$. \end{lem}
\begin{proof} For (1), if $Y$ is in $(\sC,G)$, then $G\text{-}h_\sC(W)^G(Y)$   is just the set of $G$-equivariant maps $f:W\to Y$ in $\Sch/k$, so $G\text{-}h_\sC(W)^G\cong h_{\sC,G}(W)$. 

Letting $W^{sn}\to W$ be the semi-normalization of $W$, the $G$-action on $W$ lifts uniquely to a $G$-action on $W^{sn}$ and we have $G\text{-}h_\sC(W^{sn})=G\text{-}h_\sC(W)$. Similarly, the semi-normalization of $W/G$ is $W^{sn}/G$, so we may assume that $W$ is semi-normal. 
Then $G\text{-}h_\sC(W)^G\cong h_{\sC,G}(W)$, so we have $G\text{-}h_\sC(W)^G/G\cong h_{\sC,G}(W)/G$. For $W\in (\sC,G)$, the adjoint property for $triv_\#$ gives a canonical isomorphism $h_{\sC,G}(W)/G\cong h_\sC(W/G)$; in general,  we may write $W=\amalg_iW_i$ with each $W_i$ in $(\sC,G)$,  from which follows 
$h_{\sC,G}(W)/G\cong h_\sC(W/G)$.   (2) follows from (1) applied to $W=Z^n$. 
\end{proof}
Via this lemma, we may denote the various presheaves represented by quasi-projective schemes or $G$-schemes $W$ simply by $W$, and also write $\Sym^nW$ for the presheaf $\Sym^nh_\sC(W)$,  leaving the context to determine the precise meaning. We do the same in the pointed setting.

Take $(X,x)$ a pointed quasi-projective scheme  and $A\subset X$ a reduced closed  subscheme containing $x$, giving us the pointed presheaf $X/A:=h_\sC(X)/h_\sC(A)$ on $\sC$.  We will need to relate $\Sym^n_\bullet X$ and $\Sym^n_\bullet X/A$. For this, let  $\sigma_{1,n-1}:A\times \Sym^{n-1}_\bullet X\to \Sym^n_\bullet X$
be the sum map and let $\tilde\pi_{1,n-1}: A\times \Sym^{n-1}_\bullet X\to \Sym^{n-1}_\bullet X/A$
be the projection $A\times \Sym^{n-1}X\to  \Sym^{n-1}X$ followed by the quotient map $ \Sym^{n-1}X\to 
 \Sym^{n-1}_\bullet X/A$. Since $\sigma_{1,n-1}$ is a finite morphism, we may define the pointed closed subscheme $\Sym^n_\bullet (X,A)$ of $\Sym^{n}_\bullet X$ as the reduced image of $\sigma_{1,n-1}$.

 \begin{lem} \label{lem:SymInd} Let $st_n:\Sym^{n-1}_\bullet X/A\to \Sym^n_\bullet X/A$ be the stabilization map. There is a commutative co-cartesian diagram
 \begin{equation}\label{eqn:SymInd}
 \xymatrix{
 \Sym^n_\bullet(X,A)\ar[r]^i\ar[d]_{\pi_n}&\Sym^n_\bullet X\ar[d]\\
\Sym^{n-1}_\bullet X/A\ar[r]_{st_n}& \Sym^n_\bullet X/A
}
\end{equation}
in $\Spc_\bullet(k)$, with $\pi_n\circ \sigma_{1,n-1}=\tilde\pi_{1,n-1}$.
\end{lem}

\begin{proof} Let  $(X,A)^n$ be the  reduced closed subscheme of $X^n$ of tuples $(x_1,\ldots, x_n)$ with at least one $x_i$ in $A$ and  let $(X/A,\pt)^n\subset (X/A)^n$ be the   subpresheaf of   ``points" $(y_1,\ldots, y_n)$ such that at least one of the $y_i$ is the base-point. The quotient map $X^n\to (X/A)^n$ restricted to $(X,A)^n$ defines the map
$\hat\pi_{1,n}:(X,A)^n\to (X/A,\pt)^n$.

One sees by evaluation on $Y\in \sC$ that the diagram in $\Sigma_n\text{-}\Spc^\sC_\bullet(k)$
\[
 \xymatrix{
(X,A)^n\ar[r]^-i\ar[d]_{\hat\pi_n}& X^n\ar[d]\\
(X/A,\pt)^n\ar[r]_-j&(X/A)^n}
\]
is co-cartesian and $i$ is a monomorphism. From this, it follows that for each $Y\in (\sC,\Sigma_n)$, the diagram of pointed $\Sigma_n$-sets
\[
 \xymatrix{
(X,A)^n(Y)\ar[r]^-{i_Y}\ar[d]_{\hat\pi_{nY}}& X^n(Y)\ar[d]\\
(X/A,\pt)^n(Y)\ar[r]_-{j_Y}&(X/A)^n(Y)}
\]
is co-cartesian, and $i_Y$ and $j_Y$ are monomorphisms. This implies that the diagram of ${\Sigma_n}$-invariants
\[
 \xymatrix{
(X,A)^n(Y)^{\Sigma_n}\ar[r]^-{i_Y}\ar[d]_{\hat\pi_{nY}}& X^n(Y)^{\Sigma_n}\ar[d]\\
(X/A,\pt)^n(Y)^{\Sigma_n}\ar[r]_-{j_Y}&(X/A)^n(Y)^{\Sigma_n}}
\]
is co-cartesian as well, hence the diagram 
\[
 \xymatrix{
((X,A)^n)^{\Sigma_n}\ar[r]^-i\ar[d]_{\hat\pi_n}& (X^n)^{\Sigma_n}\ar[d]\\
((X/A,\pt)^n)^{\Sigma_n}\ar[r]_-j&((X/A)^n)^{\Sigma_n}}
\]
is co-cartesian in $\Spc_\bullet^{\sC,\Sigma_n}(k)$. Applying the left adjoint $(-)/\Sigma_n$ thus gives the co-cartesian diagram
\[
 \xymatrix{
((X,A)^n)^{\Sigma_n}/\Sigma_n\ar[r]^-i\ar[d]_{\hat\pi_n}& \Sym^n_\bullet X\ar[d]\\
((X/A,\pt)^n)^{\Sigma_n}/\Sigma_n\ar[r]_-j&\Sym^n_\bullet X/A.}
\]
By lemma~\ref{lem:FiniteTypeSym}, $((X,A)^n)^{\Sigma_n}/\Sigma_n$ is represented by the pointed closed subscheme $(X,A)^n/\Sigma_n$ of $\Sym^n_\bullet X$, i.e., by $\Sym^n_\bullet(X,A)$. 

We claim that there is an isomorphism $((X/A,\pt)^n)^{\Sigma_n}/\Sigma_n\cong \Sym^{n-1}_\bullet(X/A)$ (after restricting to presheaves on $\Sm/k$) so that the map $j$ becomes the stabilization map.  

To see this, we first define a morphism $\psi:\Sym^{n-1}_\bullet(X/A)\to ((X/A,\pt)^n)^{\Sigma_n}/\Sigma_n$.  Let $f:Y\to (X/A)^{n-1}$ be a $\Sigma_{n-1}$-equivariant map with $Y\in (\sC,\Sigma_{n-1})$; write $f=(f_1, \ldots, f_{n-1})$. Letting $\Sigma_{n-1}$ act on $(X/A)^n$ via the first $n-1$ factors, extend $f$ to the $\Sigma_{n-1}$-equivariant map  $f_*:Y\to (X/A, \pt)^n$, $f_*:= (f_1, \ldots, f_{n-1}, f_n)$, with $f_n$ the constant map to the base-point.  $f_*$ induces  the $\Sigma_n$-equivariant map $\ind f:\ind_{\Sigma_{n-1}}^{\Sigma_n}Y\to (X/A, \pt)^n$, giving via \eqref{eqn:Colim} the map 
\[
\ind f/\Sigma_n:\ind_{\Sigma_{n-1}}^{\Sigma_n}Y/\Sigma_n\to ((X/A,\pt)^n)^{\Sigma_n}/\Sigma_n. 
\]
Via the isomorphism $Y/\Sigma_{n-1}\cong \ind_{\Sigma_{n-1}}^{\Sigma_n}Y/\Sigma_n$, sending $f/\Sigma_n$ to $\ind f/\Sigma_n$ passes to the colimit defining $\Sym^{n-1}_\bullet(X/A)$ via \eqref{eqn:Colim}, giving the map $\psi$.

We now define an inverse to $\psi$, but only as presheaves on $\Sm/k$. Take $Y$ in $\Sm/k$, irreducible, and let $g:Y\to ((X/A,\pt)^n)^{\Sigma_n}/\Sigma_n$ be a map in $\Spc(k)$. Using  \eqref{eqn:Colim} to describe  $((X/A,\pt)^n)^{\Sigma_n}/\Sigma_n$, there is a $Z\in (\sC,\Sigma_n)$, a $\Sigma_n$-equivariant map  $f:Z\to (X/A,\pt)^n$ and a map $\tilde{g}:Y\to Z/\Sigma_n$ representing $g$.

Let $Z_Y=[Z\times_{Z/\Sigma_n}Y]_\red$. We claim that the   map $p:Z_Y/\Sigma_n\to Y$ induced by $p_2$ is an isomorphism. Indeed, $p$ is finite, hence proper, and evidently a bijection on the underlying topological spaces. Thus $p$ is a homeomorphism and hence $Z_Y/\Sigma_n$ is irreducible; as $Z_Y$ is reduced, $Z_Y/\Sigma_n$ is integral.  Since the characteristic is zero, $p$ is birational. Since $Y$ is smooth, $p$ is an isomorphism by Zariski's main theorem.  Replacing $Z$ with $Z_Y$ and changing notation, so we may assume that $Z/\Sigma_n$ is smooth and irreducible and the map $g:Y\to Z/\Sigma_n$ is an isomorphism.

Let $\mu:Z^N\to Z$ be the normalization of $Z$. Then the $\Sigma_n$-action on $Z$ lifts to a $\Sigma_n$-action on $Z^N$, and the map on the quotients $\mu/\Sigma_n:Z^N/\Sigma_n\to Z/\Sigma_n=Y$ is thus finite and birational. As $Y$ is smooth, $Z^N/\Sigma_n\to Z/\Sigma_n$ is an isomorphism by Zariski's main theorem. Thus we may assume that $Z$ is normal; in particular, $Z$ is a disjoint union of its irreducible components. 

The map $f:Z\to (X/A,\pt)^n$ may be written as $f=(f_1,\ldots, f_n)$, $f_i:Z\to X/A$. We suppose that $f$ is not the map to the base-point. From the definition of $X/A$ as a quotient of $X$, it follows that the set of points $z\in Z$ such that $f_i(z)=\pt$ is a closed subset. Thus, for   $Z_1$ an irreducible component of $Z$, there is an $i$ such that $f_i(Z_1)=\pt$. As   $\Sigma_n$ acts on  $(X/A,\pt)^n$  by permuting the factors,   we may choose $Z_1$ so  that $f_n(Z_1)=\pt$. Letting $Z^*$ be the $\Sigma_{n-1}$-orbit of $Z_1$, where $\Sigma_{n-1}$ is identified with the subgroup of $\Sigma_n$ fixing $n$, we see that $f_n(Z^*)=\pt$ and the evident map $\ind_{\Sigma_{n-1}}^{\Sigma_n}Z^*\to Z$ is an isomorphism. This gives the isomorphism
$Z^*/\Sigma_{n-1}\cong Z/\Sigma_n\cong Y$. The $\Sigma_{n-1}$-equivariant map $(f_1,\ldots, f_{n-1}):Z^*\to (X/A)^{n-1}$ and the isomorphism $Y\cong Z^*/\Sigma_{n-1}$ gives the map $\phi_Y(g):Y\to \Sym^{n-1}_\bullet(X/A)$; in case $f(Z)=\pt$, we define $\phi_Y(g)$ to be the map to the base-point.

One checks that sending $g$ to $\phi_Y(g)$ gives a well-defined map 
\[
\phi_Y:((X/A,\pt)^n)^{\Sigma_n}/\Sigma_n(Y)\to \Sym^{n-1}_\bullet(X/A)(Y), 
\]
natural in $Y$, and thus defines a map of presheaves on $\Sm/k$,  
\[
\phi:((X/A,\pt)^n)^{\Sigma_n}/\Sigma_n\to \Sym^{n-1}_\bullet(X/A), 
\]
which is easily seen to be inverse to $\psi$,  completing the proof.
\end{proof}

Let $X$ be a finite type $k$-scheme. We have the presheaf of abelian groups on $\Sm/k$, $\Z^{tr}(X)$, with value $\Z^{tr}(X)(Y)$ the finite correspondences from $Y$ to $X$, that is, the free abelian group on integral closed subschemes $W\subset Y\times_kX$ which are finite over $Y$ and dominate an irreducible component of $Y$. Replacing the free abelian group with the monoid of sums $\sum_in_iW_i$, $n_i\ge0$, gives the subpresheaf of monoids $\Z^{tr}_\eff(X)$. We may also consider the presheaf of degree $n$ correspondences $L_n(X)$, this being the presheaf of sets on $\Sm/k$ which for irreducible $Y$ is the set of finite sums $\sum_in_iW_i$ with $\sum_in_i\Deg(W_i/Y)=n$ ($L_0(X)=
\{0\}$). For $Y=\amalg_jY_j$ with each $Y_j$ irreducible, $L_n(Y):=\prod_jL_n(Y_j)$. If $(X,x)$ is a pointed scheme, we point $L_n(X)$ with $n\cdot x\times Y$ the base-point in $L_n(X)(Y)$. 

Let $(X,x)$ be a pointed scheme in $\sC$, $A\subset X$ a reduced closed subscheme containing $x$. We let $\Z^{tr}(X/A):=\Z^{tr}(X)/\Z^{tr}(A)$. Define $\Z^{tr}_\eff(X/A)$ to be the quotient of $\Z^{tr}_\eff(X)$ by relation induced by the quotient map $\Z^{tr}(X)\to \Z^{tr}(X/A)$. Concretely, $W\sim W'$ in $\Z^{tr}_\eff(X)(Y)$ when $W-W'$ is in $\Z^{tr}(A)(Y)$. 

We have the evident isomorphisms
\[
\Z^{tr}(X)\cong \Z^{tr}_\eff(X)^+,\quad \Z^{tr}(X/A)\cong \Z^{tr}_\eff(X/A)^+,
\]
where $(-)^+$ denotes group completion.

Define quotients $q_n:L_n(X)\to L_n(X/A)$, and stabilization maps 
\[
st_n: L_{n-1}(X/A)\to L_n(X/A)
\]
inductively as follows: $L_0(X/A)=\pt$, $q_1:L_1(X)\to L_1(X/A)$ is the quotient map $X\to X/A$. Having defined these for $j=0,\ldots, n-1$, let 
$\pi_j:L_j(A)\times L_{n-j}(X)\to L_{n-1}(X/A)$ be the composition
\[
L_j(A)\times L_{n-j}(X)\xrightarrow{p_2}L_{n-j}(X)\xrightarrow{q_{n-j}}L_{n-j}(X/A)\xrightarrow{st}L_{n-1}(X/A),
\]
where $st$ is the composition of the stabilization maps.  Define $L_n(X/A)$, $q_n$ and $st_n$   by requiring the diagram 
\begin{equation}\label{eqn:CoCartDefn}
\xymatrix{
\vee_{j=1}^nL_j(A)\times L_{n-j}(X)\ar[r]^-\sigma\ar[d]_\pi&L_n(X)\ar[d]^{q_n}\\
L_{n-1}(X/A)\ar[r]_{st_n}& L_n(X/A)}
\end{equation}
to be co-cartesian.  

It follows by an easy induction that the stabilization map $st_n$  is a monomorphism. The sum maps for $L_*(X)$   induce sum maps $L_n(X/A)\times L_m(X/A)\to L_{n+m}(X/A)$. 

\begin{lem} \label{lem:SymTransfer} Let  $(X, x)$ be a pointed scheme in $\sC$, and let $A\subset X$ be a reduced closed subscheme containing $x$.\\
 1. The system of maps $L_n(X/A)\to L_{n+1}(X/A)$ and $L_n(X)\to \Z^{tr}_\eff(X)$  induce an isomorphism in $\Spc_\bullet(k)$
\[
\colim_nL_n(X/A)\to \Z^{tr}_\eff(X/A).
\]
2. We have natural isomorphisms  in $\Spc_\bullet(k)$: $L_n(X)\cong \Sym^n(X)$, $\Z^{tr}_\eff(X)\cong \Sym^\infty_\bullet(X)$,
$L_n(X/A)\cong \Sym^n_\bullet(X/A)$, $\Z^{tr}_\eff(X/A)\cong \Sym^\infty_\bullet(X/A)$.
\end{lem}

\begin{proof} We first prove (2), except for the last isomorphism. Let $\pi_n:X^n\to\Sym^n X$ be the quotient map, $\Delta_X\subset X^2$ the diagonal. Applying $\pi_n\times \id_X$ to $X^{n-1}\times\Delta_X$ gives as image the integral closed subscheme $W_n\subset (\Sym^nX)\times X$. For each morphism $f:Y\to \Sym^nX$, $Y\in \Sm/k$,  taking the pull-back cycle $(f\times\id_X)^*(W_n)$ yields an element of $L_n(X)(Y)$. By \cite[proposition 3.5]{VoevMotivic}, this defines a natural isomorphism 
\[
\Sym^n_\bullet(X)\to L_n(X)
\]
as pointed presheaves on $\Sm/k$. We have the evident isomorphism $\colim_nL_n(X)\to \Z^{tr}_\eff(X)$, which thus yields the isomorphism $\Sym^\infty_\bullet(X)\cong \Z^{tr}_\eff(X)$. 

The pointed closed subscheme $\Sym^n_\bullet(X,A)$ of $\Sym^n_\bullet(X)$ represents the union of the images of the sum maps $L_j(A)\times L_{n-j}(X)\to L_n(X)$, $j=1,\ldots, n$, via the isomorphism $\Sym^n_\bullet(X)\cong L_n(X)$. Indeed, for $Y\in \Sm/k$ irreducible, and $W$ a relative degree $n$ effective cycle on $Y\times X$,  we can write $W$ uniquely as a sum $W=W_A+W'$ with $W_A$ supported on $Y\times A$, and no component of the support of $W$ contained in $Y\times A$. If $W_A$ has degree $j$ over $Y$, then necessarily $j\le n$, and $W$ is in the image of $L_j(A)\times L_{n-j}(X)\to L_n(X)$. If $f:Y\to \Sym^nX$ is the morphism corresponding to $W$, then by considering geometric points, we see that $f(Y)\subset \Sym^n(X,A)$ if and only if $j>0$. Noting that the image of the  stabilization maps is clearly the same as the image of the monomorphism $L_{n-1}(X/A)\to L_n(X/A)$, the isomorphism $\Sym^n(X/A)\to L_n(X/A)$ follows by comparing the co-cartesian diagrams \eqref{eqn:SymInd} and \eqref{eqn:CoCartDefn} and induction.

For (1), arguing as in the last paragraph, we see that the quotient map 
\[
\Z^{tr}_\eff(X)(Y)\to \Z^{tr}_\eff(X/A)(Y)
\]
is given by the relation: $W=W_A+W'$ is equivalent to $T=T_A+T'$ if and only if $W'=T'$. This together with our description of $L_n(X/A)$ above proves (1); the last isomorphism in (2) follows from this and the isomorphisms $L_n(X/A)\cong \Sym^n_\bullet(X/A)$.
\end{proof}
\end{appendix}

\end{document}